\documentclass[14pt]{amsart}
\usepackage{cases}
\usepackage{amsmath}
\usepackage{amsfonts}
\usepackage{bm}
\usepackage{amsfonts,amsmath,amssymb,amscd,bbm,amsthm,mathrsfs,dsfont}
\usepackage{mathrsfs}
\usepackage{pb-diagram}
\usepackage{color}
\usepackage{amssymb}
\usepackage{xypic}
\usepackage[all]{xy}
\usepackage{mathrsfs}
\usepackage{amsthm}
\usepackage[all]{xy}
\usepackage{epsfig}
\usepackage{graphics}
\usepackage{array}
\usepackage{graphicx} 
\usepackage{epstopdf}
\usepackage{float}
\usepackage{tikz}
\usepackage{tikz-cd}
\usetikzlibrary{arrows}
\usetikzlibrary{graphs}
\usepackage{hyperref}
\usepackage{delarray}
\usepackage{CJK}

\usepackage{subfigure}
\usepackage{booktabs}
\usepackage{extarrows}
\usepackage{float}
\usepackage{rotating}
\usepackage{multirow}

\newtheorem{thm}{Theorem}[section]
\newtheorem{lem}[thm]{Lemma}
\newtheorem{defi}[thm]{Definition}
\newtheorem{cor}[thm]{Corollary}
\newtheorem{prop}[thm]{Proposition}

\newtheorem{thma}{Theorem}

\newtheorem{cora}[thma]{Corollary}

\title[Presentations of mapping class groups and an application ]
{Presentations of mapping class groups and an application to cluster algebras from surfaces$^*$}

\author{Jinlei Dong \;\;\;and \;\;\; Fang Li  }
\address{Jinlei Dong
	\newline
	School of Mathematical Sciences,
	Zhejiang University,
	Yuhangtang Road 866,
	Hangzhou, Zhejiang 310058,
	China P.R. }
\email{jinleidong@zju.edu.cn}

\address{Fang Li
	\newline School of Mathematical Sciences,
	Zhejiang University,
	Yuhangtang Road 866,
	Hangzhou, Zhejiang 310058,
	China P.R.}
\email{fangli@zju.edu.cn}

\date{\today}

\newcommand{\lra}{\longrightarrow}

\newcommand{\ra}{\rightarrow}
\newcommand{\sdp}{\times\kern-.2em\vrule height1.1ex depth-.05ex}
\newcommand{\epi}{\lra \kern-.8em\ra}

\makeatletter
\@addtoreset{equation}{section}
\makeatother

 \normalbaselines
\begin{document}
	
	\renewcommand{\thefootnote}{\alph{footnote}}
	\setcounter{footnote}{-1} \footnote{*Project supported by the National Natural Science Foundation of China (No.12071422 and No.12131015). }
	\setcounter{footnote}{-1}
\setcounter{footnote}{-1}

	\renewcommand{\thefootnote}{\alph{footnote}}
	\maketitle
	\bigskip
	
	\begin{abstract}
 In this paper, we give presentations of the mapping class groups of marked surfaces stabilizing boundaries for any genus. Note that in the existing works,  the mapping class groups of marked surfaces were the isotopy classes of homeomorphisms fixing boundaries pointwise. The condition for stabilizing boundaries of mapping class groups makes the requirement for mapping class groups to fix boundaries pointwise to be unnecessary.

As an application of  presentations of the mapping class groups of marked surfaces stabilizing boundaries, we obtain the presentation  of the cluster automorphism group  of a cluster algebra from a feasible surface $ (S,M) $.

 Lastly, for the case (1) 4-punctured sphere, the cluster automorphism group of a cluster algebra from the surface is characterized. Since cluster automorphism groups of cluster algebras from those surfaces were given in  \cite{ASS} in the cases (2) the once-punctured 4-gon and (3) the twice-punctured digon, we indeed give presentations of cluster automorphism groups of  cluster algebras from surfaces which are not feasible.
	\end{abstract}

	\tableofcontents

	\section{Introduction}

Mapping class groups are definded as groups of isotopy classes of homeomorphisms of topological spaces. It is an important algebraic invariant of topological spaces, which is related to many fields in mathematics such as hyperbolic geometry, symplectic geometry and dynamics.

Mapping class groups of oriented surfaces have been heavily studied, and it is closely related to Teichm\"{u}ller spaces and moduli spaces. It is well known that mapping class groups of  closed oriented surfaces are generated by Dehn twists, and there are lots of works on presentations of mapping class groups.

In the early work, Hatcher and Thurston in \cite{HT} gave a presentation of the mapping class group of a closed oriented surface, via the action of the mapping class group on a simply connected complex introduced by them. Before that, based on some results of Hatcher and Thurston, Wajnryb in \cite{W} gave a new presentation of the mapping class group of a closed oriented surface or an oriented surface with one boundary component, which involves less generators and relations.
 In \cite{M}, Matsumoto presented the mapping class group of an oriented surface with one boundary component as the quotient of an Artin group with the relations given in terms of fundamental elements of the Artin group.  Based on the work of Matsumoto, the authors of \cite{LP} gave a presentation of the mapping class group of any oriented bordered surface with genus greater than zero and finite punctures in terms of Artin groups.	

Note that in the works mentioned above, the mapping class groups of surfaces are isotopy classes of homeomorphisms fixing boundaries pointwise.
But in this article, this requirement is not needed. The surface discussed by us is equipped with finite marked points such that each boundary component contains at least one marked point. Mapping class groups here should map marked points to themselves but is not required to fix boundaries pointwise.

Based on the presentations of mapping class groups of surfaces in the case with genus greater than zero in \cite{LP} and in the case with genus equal to zero in \cite{FM}, we give presentations mapping class groups of oriented bordered surfaces with finite marked points  in general not fixing boundaries pointwise.
Denoting by $ \mathcal{P} $ the set of punctures and $ \mathcal{B}_{i}, i \in \mathbb{Z}_{>0}$ the set of boundary components which contain $ i $ marked points. In this case, $ \mathcal{P} $ and $ \mathcal{B}_{i} $ are stable under the actions of these mapping class groups, respectively.  Thus we say them  to be {\bf mapping class groups stabilizing boundaries} and denote the groups  as $ \mathcal{MCS}(S,M) $.

\begin{thma}[Theorem \ref{mcgSg0prop}]
		Suppose $S=S_{g,r} $ has $ n_{k} $ boundary components with $ k $ marked points, the set $ \mathcal{B} $ of boundary components, the set $ \mathcal{P}_{n} $ of punctures, the set $ M $ of marked points.
		Let $ \tilde{\sigma}_{k} $, $ \tilde{a}_{ij} $ be lifts of $ \sigma_{k} $ and $ a_{ij} $ as in the presentation of $ \text{\em Im}\pi $ in Lemma \ref{Impresentation}, respectively, $ T_{b_{l}}  $ the $ 1/m $ twist about the boundary component $ b_{l} $, and $ I = [n+r-1]\backslash \{n, n+\Sigma_{i=1}^{j}n_{i} | j\in \mathbb{N}  \}  $.
	
		Then for the case  $ (S,\mathcal{P}_{n}) \neq( S_{0,2}, \emptyset)  $, if the genus $ g=0 $, then $  \mathcal{MCG}(S,M)  $ is generated by $ \{ \tilde{\sigma}_{k}, \tilde{a}_{ij}, T_{b_{l}} |  k \in I; 1\le i \le n+r; b_{l} \in \mathcal{B} \} $ with the relations (\ref{g0relation1}) and (\ref{g0relation2}).
\end{thma}

\begin{thma}[Theorem \ref{mcgSg1prop}]
		  For genus $ g \ge 1 $,  suppose  $S=S_{g,r} $ has $ n_{i} $ boundary components with $ i $ marked points, the set $  \mathcal{B} $ of boundary components, the set of punctures
		 $ \mathcal{P}_{m} $ and the set $ M $ of marked points, the set $  \mathcal{B} $ is ordered by the number of marked points as given in (\ref{order}) with a map keeping order: $ \omega : \mathcal{P}_{m} \cup \mathcal{B}  \to  \mathcal{P}_{m+r}$ for $m+r=n  $,  and
		 $ I = [n-1]\backslash \{m, m+\Sigma_{i=1}^{j}n_{i} | j\in \mathbb{N}  \}  $.
		
		Let $ A(\Gamma_{g,0,n}) $ be the Artin group associated with $ \Gamma_{g,0,n} $ (see Figure \ref{Sr0coxgrafig}), $ A' $ the subgroup generated by $ x_{i}, 0 \le i \le n $, $ y_{j}, j \in [2g-1] $, $ z $ and $ v_{k} $, $ k \in I $.  Suppose $ T_{b_{l}} ,b_{l} \in \mathcal{B} $ are $ 1/m $ twists about $ b_{l} $. Then $ \mathcal{MCG}(S,M) $ is isomorphic to the quotient of $ A'\langle T_{b_{l}}, b_{l} \in \mathcal{B} \rangle $ with the relations (\ref{g1relation1}) and (\ref{g1relation2}).
	
\end{thma}

Cluster algebras defined by Fomin and Zelevinsky are closely related to many fields of mathmatics, such as canonical bases and total positivity. Cluster automorphisms corresponding to cluster algebras introduced by  \cite{ASS} are their automorphisms of algebras compatible with the structure of cluster algebras. According to the connection between
mapping class groups for surfaces and cluster automorphism groups of cluster algebras from these surfaces given by \cite{Gu,TBIS,BY}, we give the presentations of cluster automorphism groups of cluster algebras from feasible surfaces, that is, which are different from any one of (1) the 4-punctured sphere, 		
(2) the once-punctured 4-gon, 		
(3) the twice-punctured digon.

\begin{cora}[Corollary \ref{preclusterautocor}]
	For a feasible surface $(S,M)$,
		suppose $S=S_{g,r} $ has the set $ \mathcal{P}_{n} $ of punctures and the set $ M $ of marked points, and  $ (S,\mathcal{P}_{n}) \neq( S_{0,2}, \emptyset)  $.
		
			Then $\text{\em Aut}\mathcal A(S,M)$  is given as follows:
	
		(a)  if $ (S,M) $ is a once-punctured closed surface, then
		$ \text{\em Aut}\mathcal{A}(S,M) \cong  \mathcal{MCG}(S,M)\rtimes Z_{2},$	
	
		(b)  if $ (S,M) $ is not a once-punctured closed surface, then
		$ \text{\em Aut}\mathcal{A}(S,M) \cong (\mathcal{MCG}(S,M)\rtimes Z_{2}) \ltimes \mathbb{Z}_{2}^{\mathcal{P}_{n}},$
	where  the presentation of $ \mathcal{MCG}(S,M) $ is given by Theorem \ref{mcgSg0prop} for case genus $ g=0 $ and by Theorem \ref{mcgSg1prop} for case genus $ g \ge 0 $.
\end{cora}

Since presentations of the cluster automorphism groups in cases (2) and (3) are given by \cite{ASS}, we will only give a presentation of the cluster automorphism group in the case (1) in above.
\begin{thma}[Theorem \ref{Precase1}]
	Suppose $ \mathcal{A}(S,M) $ is the cluster algebra from the 4-punctured sphere, then
		\begin{equation*}
			\text{\em Aut}(S,M) \cong  (\mathcal{MCG}(S,M)\rtimes Z_{2}) \ltimes \mathbb{Z}_{2}^{\mathcal{P}_{4}} \times \mathbb{Z}_{2}^{2}.
	\end{equation*}
\end{thma}
 So, in fcat, we have complete the presentations of cluster automorphism groups of cluster algebras from surfaces.

\section{Preliminaries}

In this section, we introduce some background materials about mapping class groups and their generators, and Artin groups. We also give presentations of mapping class group fixing boundaries for genus $ g=0 $ and $ g \ge 1 $.

\subsection{Marked surfaces, mapping class groups and twists  of an oriented surface} \quad

A \textbf{marked surface} is a pair $ (S,M) $, where $ S $ is an oriented Riemannian surfaces and $ M  $ consists of finite points of $ S $ called {\bf marked points} of $ S $ such that  each boundary component contains at least one marked point,  We also call $ (S,M) $ a {\bf surface} for short. The points in $ M \backslash \partial S $ are called {\bf punctures}, where $ \partial S $ is the boundary of $ S $.

It should be noted that all surfaces $(S,M)$ mentioned in this article are connected oriented  surfaces $ S$ with finite boundary components.

\begin{defi}
	An \textbf{arc} $ \gamma $ in $ (S,M) $ is a curve (up to isotopy) in $ S $ such that
	
	(i) the endpoints  of $\gamma  $ lie in $ M $;
	
	(ii) $ \gamma $ does not intersect with boundary of $ S $ except endpoints;
	
	(iii) $ \gamma $ does not cut out an unpuctured monogon or an unpunctured digon;
	
	(iv) $ \gamma $ does not intersect with itself.
\end{defi}

If a curve $  \gamma $ satisfies (i), (ii), (iii),  then $  \gamma $ is called a \textbf{generalized arc}.

Two arc in $ (S,M) $ are called \textbf{compatible} if there are curves in their respective isotopy classes do not intersect. And a maximal collection of distinct pairwise compatible arcs is called an \textbf{ideal triangulation}.

\begin{defi}
	An \textbf{tagged arc} is an arc in which each endpoint has been tagged plain or notched, so that the following conditions are satisfied:
	
	(i) the arc does not cut out a once-punctured monogon;
	
	(ii) all endpoints lying on the boundary are tagged plain;
	
	(iii) if its endpoints are the same, they are tagged in the same way.
\end{defi}
Two tagged arcs $ \alpha $ and $ \beta $ are called \textbf{compatible} if the plain arcs $ \bar{\alpha} $, $ \bar{\beta} $ obtained by forgetting their taggings are compatible and satisfy

(a) if  $ \bar{\alpha}  = \bar{\beta}$, then at least one endpoint of $ \bar{\alpha} $ and $ \bar{\beta} $ is tagged in the same way;

(b) if  $ \bar{\alpha}  \neq \bar{\beta}$, but they have common endpoints, then they are tagged in the same way at common endpoints.

Similarly, a maximal collection of distinct pairwise compatible tagged arcs in $ (S,M) $ is called a \textbf{tagged triangulation} of $ (S,M) $.

In fact, each ideal triangulation (resp. tagged triangulation)  has the same number of tagged arcs (resp. tagged arcs).

Suppose $ \mathcal{T} $ is a triangulation (resp. tagged triangulation) of $ (S,M) $, and $ \gamma \in \mathcal{T} $. Then there is another arc (resp. tagged arc) $ \gamma' $ such that $ \mathcal{T}' =(T \backslash \{ \gamma \}) \cup \{ \gamma' \} $ is another triangulation (resp. tagged triangulation) of $ (S,M) $. And the way to obtain $ \mathcal{T}' $ from $ \mathcal{T} $ is called the \textbf{flip} at $ \gamma $.

\begin{defi}\cite{BY} \label{defmcg}
	The \textbf{mapping class group stabilizing boundaries} for a bordered surface with marked points $ (S,M) $ is defined by
	\[  \mathcal{MCG}(S,M) := \text{\em Homeo}^{+}(S,M)/\text{\em Homeo}^{+}_{0}(S,M) ,\]
	where $ \text{\em Homeo}^{+}(S,M) $ is the group of orientation-preserving homeomorphisms such that $ g(M)=M $ for any  $ g \in  \text{\em Homeo}^{+}(S,M) $, and $ \text{\em Homeo}^{+}_{0}(S,M) $ is its subgroup of homeomorphisms homotopic to the identity.
	
	Let $ \mathcal{P}_{n}= \{ p_{1}, p_{2}, ..., p_{n} \}$   be the set of punctures of $(S,M)  $, the \textbf{tagged mapping class group} of $ (S,M) $ is defined to be the semidirect product
	\[ \mathcal{MCG}_{\bowtie}(S,M) :=  \mathcal{MCG}(S,M) \ltimes \mathbb{Z}_{2}^{\mathcal{P}_n},  \]
	such that
	\begin{equation}\label{semiprodpunctures}
		(g_{1},R_{1})*(g_{2},R_{2}):=(g_{1}g_{2},g(R_{1})\ominus R_{2}), \;\;\;  (g_{1},R_{1}),(g_{2},R_{2}) \in \mathcal{MCG}_{\bowtie}(S,M).
	\end{equation}
\end{defi}
 Note that each element in $ \mathbb{Z}_{2}^{\mathcal{P}_n} $ can be identified as a subset of $ \mathcal{P}_{n} $, and its product is given by $ R_{1}\ominus R_{2} := R_{1}\cup R_{2} \backslash R_{1}\cap R_{2} $.

We can also define an unsigned version of the mapping class group. Denote
by $ \text{Homeo}(S,M) $ the group of homeomorphisms of $  S $ that takes $ M $ to itself. Define
\[ \mathcal{MCG}^{\pm}(S,M) := \text{Homeo}(S,M)/\text{Homeo}^{+}_{0}(S,M) ;  \]
\[  \mathcal{MCG}^{\pm}_{\bowtie}(S,M) :=  \mathcal{MCG}^{\pm}(S,M) \ltimes \mathbb{Z}_{2}^{\mathcal{P}_{n}}.\]

 In \cite{FM}, The mapping  class group is required to fix the boundary pointwise.

\begin{defi}\cite{FM}
	The \textbf{mapping class group fixing boundaries} for a marked surface $ (S,M) $ is defined by
		\[  \text{\em Mod}(S,M) := \text{\em Homeo}^{+}((S,M),\partial S)/\text{\em Homeo}^{+}_{0}(S,M) ,\]
where $ \text{\em Homeo}^{+}((S,M),\partial S) $ is the subgroup of $ \text{\em Homeo}^{+}(S,M) $ fixing boundaries pointwise and $ \text{\em Homeo}^{+}(S,M)$, $\text{\em Homeo}^{+}_{0}(S,M)   $ follow Definition \ref{defmcg}.
\end{defi}

Denoting by $ S_{g,r} $ an oriented surface of  genus $ g $  with $ r $ boundary components,  $ \mathcal{P}_{n}= \{ p_{1}, p_{2}, ..., p_{n}  \}$  the set of  punctures.

We can restrict the element of $\text{Mod}(S_{g,r},\mathcal{P}_{n})  $ to $\mathcal{P}_{n}  $, which is a permutation on $ \mathcal{P}_{n} $, and we obtain the following exact sequence:
\begin{equation}\label{pureexsq}
	1 \rightarrow \text{PMod}(S_{g,r},\mathcal{P}_{n}) \rightarrow \text{Mod}(S_{g,r},\mathcal{P}_{n}) \xrightarrow{\theta} \Sigma_{n}  \rightarrow 1,
\end{equation}
where $ \Sigma_{n} $ is the permutation group on $ n $. The subgroup $ \text{PMod}(S_{g,r},\mathcal{P}_{n}) $ is called the \textbf{pure mapping class group fixing boundaries}.

In \cite{FM,LP}, authors give presentations of mapping class group fixing boundaries,  $ \text{Mod}(S_{g,r},\mathcal{P}_{n}) $  is generated by half twists for $ g=0 $ or Dehn twists for genus $ g \ge 1 $.

\begin{defi}
	(i) For the circle $S^1$, suppose $A= S^1 \times [0,1] $.
	Giving the coordinates $(e^{i\theta},t)$ of any point $P$ in $A$ with $\theta \in [0, 2\pi]$ and $t \in [0,1]$, a homeomorphism $ f $ is defined as follows:
	$$ f(e^{i\theta},t)= (e^{i(\theta + 2\pi t)}, t) .$$
	Let $ \zeta $ be a loop in $ (S,M) $, $ N $ is a regular neighborhood of $ \zeta $ which is homeomorphic to $ A $ and  denote by $ g $ the homeomorphism. Then a \textbf{Dehn twist} along $ \zeta $ is defined as $ g^{-1}fg $.

	(ii)	A \textbf{half twist} $ T_{h} $ about two punctures along an arc $ \gamma $ connecting them is defined as follows:
	Choosing a loop $ \zeta $ surrounding only $ \gamma $, $ T_{h} $ is the homeomorphism of a regular neighborhood of $ \zeta $ such that $ T_{h}^{2} $ is the Dehn twist along $ \zeta $, see Figure \ref{htp}.
\end{defi}	

\begin{figure}[H]
	\centering
\begingroup%
  \makeatletter%
  \providecommand\color[2][]{%
    \errmessage{(Inkscape) Color is used for the text in Inkscape, but the package 'color.sty' is not loaded}%
    \renewcommand\color[2][]{}%
  }%
  \providecommand\transparent[1]{%
    \errmessage{(Inkscape) Transparency is used (non-zero) for the text in Inkscape, but the package 'transparent.sty' is not loaded}%
    \renewcommand\transparent[1]{}%
  }%
  \providecommand\rotatebox[2]{#2}%
  \newcommand*\fsize{\dimexpr\f@size pt\relax}%
  \newcommand*\lineheight[1]{\fontsize{\fsize}{#1\fsize}\selectfont}%
  \ifx\svgwidth\undefined%
    \setlength{\unitlength}{946.12059581bp}%
    \ifx\svgscale\undefined%
      \relax%
    \else%
      \setlength{\unitlength}{\unitlength * \real{\svgscale}}%
    \fi%
  \else%
    \setlength{\unitlength}{\svgwidth}%
  \fi%
  \global\let\svgwidth\undefined%
  \global\let\svgscale\undefined%
  \makeatother%
  \begin{picture}(1,0.06)%
    \lineheight{1}%
    \setlength\tabcolsep{0pt}%
    \put(0.07,0){\includegraphics[scale =0.3]{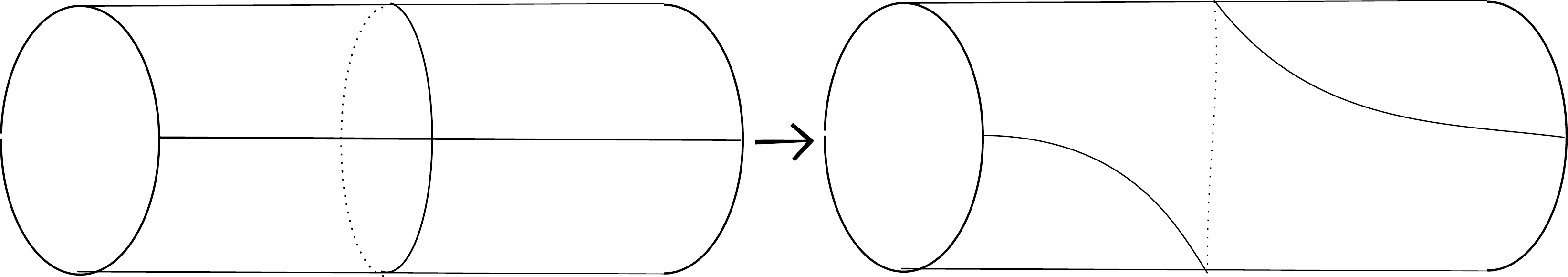}}%
    \put(0.155,0.04){\makebox(0,0)[lt]{\lineheight{1.25}\smash{\begin{tabular}[t]{l}$\zeta $\end{tabular}}}}%
  \end{picture}%
\endgroup%

	\caption{Dehn twist along $ \zeta $}
\end{figure}

\begin{figure}[H]
	\centering
	\begin{tikzpicture}[scale=1]
		\draw[dashed] (0,0) circle(0.5) circle(1)  (3,0) circle(0.5) circle(1) ;
		\draw (0,0) circle (0.75);
		
		\filldraw  (-0.5,0) circle [radius=1pt]
		(0.5,0) circle [radius=1pt]
		(2.5,0) circle [radius=1pt]
		(3.5,0) circle [radius=1pt];
		
		\draw[dashed] (-1,0) -- (-0.5,0)  (1,0) -- (0.5,0);
		
		\draw (-0.5,0)--(0.5,0)  (2.5,0)--(3.5,0);
		\draw[->] (1.3,0) -- (1.7,0);
		
		\draw[dashed] (4,0) ..controls (3.2,1.2) and (1.5,0.3) .. (2.5,0)
		(2,0) ..controls (2.8,-1.2) and (4.5,-0.3) .. (3.5,0);
		
		\node(ze) at (0,0.85) {$ \zeta$};
		\node(gm) at (0,-0.2) {$ \gamma$};
		\node(gm) at (3,-0.2) {$ \gamma$};
	\end{tikzpicture}
	\caption{Half twist about punctures along $ \gamma $}
	\label{htp}
\end{figure}
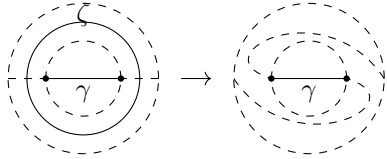
Next, we list some presentation of $  \text{Mod}(S_{g,r},\mathcal{P}_{n})$.

\subsection{Presentations of mapping class groups fixing boundaries for genus $ g=0 $} \quad

The mapping class group fixing boundaries of a  disk with $n$ punctures $ (S_{0,1},\mathcal{P}_{n}) $ is the \textbf{braid groups} $ B_{n} $:
\begin{equation}\label{braidpre}
	\begin{array}{lll}
		B_{n}= \langle \sigma_{1}, ..., \sigma_{n-1} | & \sigma_{i}\sigma_{i+1}\sigma_{i}= \sigma_{i+1}\sigma_{i}\sigma_{i+1}, & i\in [n-1], \\
		& \sigma_{i}\sigma_{j}=\sigma_{j}\sigma_{i}, & |i-j|>1 \rangle.
	\end{array}
\end{equation}
Let $ \gamma_{i} $ be arcs connecting $ p_{i}, p_{i+1} $ such that they do not intersect with each other. Then for each $ i \in [n-1] $ $ \sigma_{i} $ is the half twist along $ \gamma_{i} $.

In particular, the pure mapping class group fixing boundaries $ \text{P}B_{n}=\text{PMod}(S_{0,1},\mathcal{P}_{n}) $ of $ (S_{0,1},\mathcal{P}_{n}) $ is
\begin{equation}\label{pbn}
	\begin{array}{lll}
		\text{P}B_{n}= \langle a_{ij} | & [a_{pq},a_{rs}]=1,  [a_{ps},a_{qr}]=1, & p <q <r < s\\
		&a_{pr}a_{qr}a_{pq}=a_{qr}a_{pq}a_{pr}=a_{pq}a_{pr}a_{qr}, &  p <q <r \\
		&    [a_{rs}a_{pr}a_{rs}^{-1}, a_{qs}]=1,              &  p <q <r < s  \rangle,
	\end{array}
\end{equation}
where $ a_{ij} $ is the Dehn twist along $ \alpha_{ij} $, see Figure \ref{pbngenerator} and $ [a,b]=aba^{-1}b^{-1} $.
Note that $$ a_{ij}= \sigma_{j-1}\cdots\sigma_{i+1}\sigma_{i}^{2}(\sigma_{j-1}\cdots\sigma_{i+1})^{-1}. $$

\begin{figure}[htbp]
		\begin{minipage}[t]{0.5\linewidth}
			\centering
			\begin{tikzpicture}[scale=2]
				\draw (0,0) ellipse (2 and 1);
				
				\filldraw  (-1,0.5) circle [radius=1pt]
				(-0.3,0.5) circle [radius=1pt]
				(0.3,0.5) circle [radius=1pt]
				(1,0.5) circle [radius=1pt];
				
				\draw[dotted] (0.4,0.5) -- (0.9,0.5) (-0.9,0.5) -- (-0.4,0.5) (-0.2,0.5) -- (0.2,0.5) ;
				
				\draw (-0.3,0.7) ..controls (-0.1,0.8) and (-0.2,0.1)  .. (0,0.2)
				(-0.3,0.7) ..controls (-0.6,0.8) and (-0.3,-0.1)  .. (0,0)
				(0.3,0.7) ..controls (0.1,0.8) and (0.2,0.1)  .. (0,0.2)
				(0.3,0.7) ..controls (0.6,0.8) and (0.3,-0.1)  .. (0,0);

				\node(p1) at (-1,0.58) {$ p_{1} $};
				\node(pi) at (-0.3,0.58) {$ p_{i} $};
				\node(pj) at (0.3,0.58) {$ p_{j} $};
				\node(pn) at (1,0.58) {$ p_{n} $};
				\node(aij) at (0,-0.13) {$ \alpha_{ij} $};
				
			\end{tikzpicture}
			\caption{Generators of $ \text{P}B_{n} $}	
			\label{pbngenerator}
		\end{minipage}%
		\begin{minipage}[t]{0.5\linewidth}
			\centering
			\begin{tikzpicture}[scale=1]
				\draw (0,0) circle (2);
				\draw[dashed] (0,0) ellipse (2 and 1 );
				\filldraw  (-1,0.5) circle [radius=1pt]
				(-0.5,0.5) circle [radius=1pt]
				(0,0.5) circle [radius=1pt]
				(1,0.5) circle [radius=1pt];
				
				\draw[dotted] (0.2,0.5) -- (0.8,0.5);
				\draw (-1,0.5)--(-0.5,0.5)   (-0.5,0.5)--(0,0.5)   (0,0.5)--(0.2,0.5)   (0.8,0.5)--(1,0.5);
				
				\node(p1) at (-1,0.65) {$ p_{1} $};
				\node(p2) at (-0.5,0.65) {$ p_{2} $};
				\node(p3) at (0,0.65) {$ p_{3} $};
				\node(pn) at (1,0.65) {$ p_{n} $};
				
				\node(gm1) at (-0.75,0.3) {$ \gamma_{1} $};
				\node(gm2) at (-0.25,0.3) {$ \gamma_{2} $};
				\node(gm3) at (0.1,0.3) {$ \gamma_{3} $};
				\node(gmn-1) at (0.75,0.3) {$ \gamma_{n-1} $};
			\end{tikzpicture}
			\caption{Generators of $ \text{Mod}(S_{0,0},\mathcal{P}_{n}) $ }	
			\label{sphht}
		\end{minipage}%
	\centering
\end{figure}

\cite{FM} also give presentation of $ \text{ Mod}(S_{0,0},\mathcal{P}_{n}) $.

\begin{thm}\cite{FM} \label{S0rpresentation}
A presentation of $ \text{\em Mod}(S_{0,0},\mathcal{P}_{n}) $ is
\[ 	\begin{array}{ll}
\text{\em Mod}(S_{0,0},\mathcal{P}_{n})=\langle\sigma_{1}, ..., \sigma_{n-1}\;\;  |
  &  \sigma_{i}\sigma_{j}=\sigma_{j}\sigma_{i} \; \; \; |i-j|>1,\\
&  \sigma_{i}\sigma_{i+1}\sigma_{i}=\sigma_{i+1}\sigma_{i}\sigma_{i+1}, \\
&(\sigma_{1}\cdots\sigma_{n-1})^{n}=1 ,\\
&\sigma_{1}\cdots \sigma_{n-1}\sigma_{n-1} \cdots \sigma_{1} =1 \rangle,
\end{array}
    \]
where $ \sigma_{i}, i=1, ..., n-1 $ are half twist along $ \gamma_{i} $, see Figure \ref{sphht}.
\end{thm}

\subsection{Artin groups and presentations  of mapping class group fixing boundaries for genus $ g \ge 1 $} \quad

If $g > 0$, \cite{LP} gives a presentation of $\text{Mod}(S_{g,r},M)$ in terms of Artin groups.

A \textbf{Coxeter matrix} is an integer matrix $M=(m_{ij})_{n \times n}$ such that

\textbullet \,  $m_{ii}=1, i\in [n];$

\textbullet \,  $m_{ij} \ge 2, i,j \in [n], i\neq j$.

A coxeter matrix can be represented by a coxeter graph $\Gamma$ which is defined as follows:

\textbullet \,  $\Gamma$ has $n$ vertices: $x_1, ..., x_n$;

\textbullet \,  two vertices $x_i$ and $x_j$ are joined by an edge if $m_{ij} \ge 3$;

\textbullet \,  the edge joining two vertices $x_i$ and $x_j$ is labelled by $m_{ij}$, if $m_{ij} \ge 4$.

For $ i,j\in [n] $, we write:
\[  \text{prod}(x_{i},x_{j},m_{ij})=\left\{  \begin{array}{ll}
	(x_{i}x_{j})^{m_{ij}/2},  & m_{ij} \text{ is even} \\
	(x_{i}x_{j})^{(m_{ij}-1)/2}x_{i}. &  m_{ij} \text{ is odd}
\end{array}\right.  \]

The \textbf{Artin group} $ A(\Gamma) $ associated with $ \Gamma $ (or $ M $) is the group as follows:
\[  A(\Gamma)= \langle x_{1}, ..., x_{n} | \text{prod}(x_{i},x_{j},m_{ij}) =\text{prod}(x_{j},x_{i},m_{ji}), i\neq j \rangle. \]

Let $ X  $ be a subset of $ \{x_{1}, ..., x_{n}  \} $, and we  denote by $ \Gamma_{X} $ the Coxeter subgraph of $ \Gamma $ generated by $ X $. The \textbf{quasi-center} of an Artin group $ A(\Gamma_{X}) $ is defined to be the subgroup of elements $ \alpha $ in $  A(\Gamma_{X}) $ such that $ \alpha X \alpha^{-1}=X $. If $ \Gamma_{X} $ is of type as in Figure \ref{typecgfig}, then the quasi-center is an infinite cyclicgroup generated by a special element in $  A(\Gamma_{X}) $, called \textbf{fundamental element} and denoted by $ \Delta(X) $ (see \cite{BS}).

\begin{figure}[htbp]
	\centering
\begin{tikzpicture}[scale=0.7]
	\node(An) at (0,4) {$ A_{n} $};
	\node(Dn) at (0,2) {$ D_{n} $};
	\node(E6) at (0,0) {$ E_{6} $};
	
	\node(Bn) at (8,4) {$ B_{n} $};
	\node(E7) at (8,0) {$ E_{7} $};
	
	\node(ax1) at (1,4) {$ x_{1} $};
	\node(ax2) at (2,4) {$ x_{2} $};
	\node(axn) at (5,4) {$ x_{n} $};

	\node(dx1) at (1,2.5) {$ x_{1} $};
	\node(dx2) at (1,1.5) {$ x_{2} $};
	\node(dx3) at (2,2) {$ x_{3} $};
	\node(dx4) at (3,2) {$ x_{4} $};
	\node(dxn) at (6,2) {$ x_{n} $};
	
	\node(e6x1) at (1,0) {$ x_{1} $};
	\node(e6x2) at (2,0) {$ x_{2} $};
	\node(e6x3) at (3,0) {$ x_{3} $};
	\node(e6x4) at (4,0) {$ x_{4} $};
	\node(e6x5) at (5,0) {$ x_{5} $};
	\node(e6x6) at (3,-1) {$ x_{6} $};
	
	\node(bx1) at (9,4) {$ x_{1} $};
	\node(bx2) at (10,4) {$ x_{2} $};
	\node(bx3) at (11,4) {$ x_{3} $};
	\node(bxn) at (14,4) {$ x_{n} $};
	
	\node(e7x1) at (9,0) {$ x_{1} $};
	\node(e7x2) at (10,0) {$ x_{2} $};
	\node(e7x3) at (11,0) {$ x_{3} $};
	\node(e7x4) at (12,0) {$ x_{4} $};
	\node(e7x5) at (13,0) {$ x_{5} $};
	\node(e7x6) at (14,0) {$ x_{6} $};
	\node(e7x7) at (12,-1) {$ x_{7} $};
	
	\draw[dashed] (3,4) -- (4,4) (4,2) -- (5,2) (12,4) --(13,4);
	\draw (1.25,4) -- (1.75,4) (2.25,4) -- (3,4)  (4,4)--(4.75,4)
	
	(1.25,2.5)--(1.75,2.1) (1.25,1.5)--(1.75,1.9) (2.25,2) -- (2.75,2) (3.25,2)--(4,2) (5,2)--(5.75,2)
	
	(1.25,0) -- (1.75,0) (2.25,0) -- (2.75,0)  (3.25,0)--(3.75,0) (4.25,0)--(4.75,0) (3,-0.25)--(3,-0.75)
	
	(9.25,4) -- (9.75,4) (10.25,4) -- (10.75,4)  (11.25,4) -- (12,4) (13,4) -- (13.75,4)
	
	(9.25,0) -- (9.75,0) (10.25,0) -- (10.75,0)  (11.25,0) -- (11.75,0) (12.25,0) -- (12.75,0) (13.25,0) -- (13.75,0) (12,-0.25) -- (12,-0.75);
	
	\node(num) at (9.5,4.3) {$ 4 $};
\end{tikzpicture}
	\caption{Some types Coxeter graphs}
	\label{typecgfig}
\end{figure}
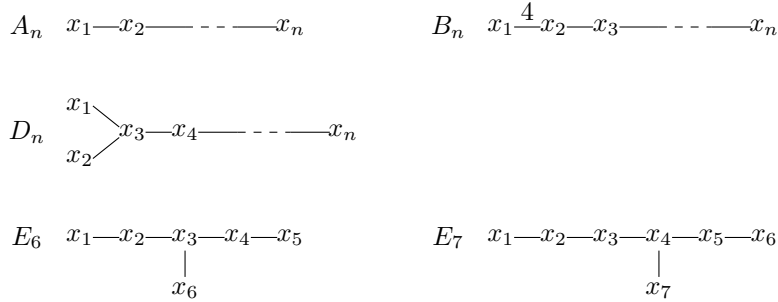

For types as in Figure \ref{typecgfig}, the center of $ A(\Gamma_{X}) $ is an infinite cyclic group generated by $  \Delta(X) $ if $ \Gamma_{X} $ is $ B_{n} $, $D_{n} $ (n even), $ E_{7} $ and by $ \Delta^{2}(X) $ if $ \Gamma_{X} $ is $ A_{n} $, $D_{n} $ (n odd), $ E_{6} $. There are explicit expressions of $ \Delta(X) $ or $ \Delta^{2}(X) $ as follows:
\begin{prop}\cite{BS}
For types as in Figure \ref{typecgfig},	the explicit expressions of $ \Delta(X) $ or $ \Delta^{2}(X) $ are
\begin{equation*}
\begin{array}{c}
		\Delta^{2}(A_{n})=(x_{1}x_{2}\cdots x_{n})^{n+1},\;\;\;\;
	\Delta(B_{n})=(x_{1}x_{2}\cdots x_{n})^{n},\\
	\Delta(D_{2p})=(x_{1}x_{2}\cdots x_{2p})^{2p-1},\;\;\;\;
	\Delta^{2}(D_{2p+1})=(x_{1}x_{2}\cdots x_{2p+1})^{4p},\\
	\Delta^{2}(E_{6})=(x_{1}x_{2}\cdots x_{6})^{12},\;\;\;\;
	\Delta(E_{7})=(x_{1}x_{2}\cdots x_{7})^{15}.
\end{array}
\end{equation*}
\end{prop}

The Perron-Vannier representation or graph representation of the Artin group of types in Figure \ref{typecgfig} is a homonorphism $ \rho $ from the Artin group to some mapping class group fixing boundaries. The homonorphism $ \rho $ sends generators $ x_{i} $ on the Dehn twist about $ \alpha_{i} $, or half twist about $ \tau_{i} $, see Figure \ref{PVrepABDEfig} and Figure \ref{GrepBlfig}. and we have the following proposition:

\begin{prop}\cite{LP} \label{funelmtoboundlem}
	Let $ \rho $ be the homonorphism given above, then we have
	\begin{equation*}
		\begin{array}{c}
			\rho(\Delta^{2}(A_{2p+1}))=T_{b_{1}}T_{b_{2}},\;\;\;\;
			\rho(\Delta^{4}(A_{2p}))=T_{b_{1}},\\
			\rho(\Delta(B_{2p}))=T_{b_{1}}T_{b_{2}},\;\;\;\;
			\rho(\Delta^{2}(B_{2p+1}))=T_{b_{1}},\\
			\rho(\Delta^{2}(D_{2p+1}))=T_{b_{1}}T_{b_{2}}^{2p-1},\;\;\;\;
			\rho(\Delta(D_{2p}))=T_{b_{1}}T_{b_{2}}b_{2}^{p-1},\\
			\rho(\Delta^{2}(E_{6}))=T_{b_{1}},\;\;\;\;
			\rho(\Delta^{2}(E_{7}))=T_{b_{1}}T_{b_{2}}^{2},
		\end{array}
	\end{equation*}
for Perron-Vannier representations, see Figure \ref{PVrepABDEfig}, and
	\begin{equation*}
		\rho(\Delta(B_{l}))=T_{b_{1}}^{l-1}T_{b_{2}},
\end{equation*}
for graph representations, see Figure \ref{GrepBlfig}. Here $ T_{b_{i}} $ is the Dehn twist about boundary component $ b_{i} $ in Figure \ref{PVrepABDEfig} and Figure \ref{GrepBlfig}.
\end{prop}

\begin{figure}[htbp]
	\centering
	\subfigure{
		\begin{minipage}[t]{0.5\linewidth}
			
\begingroup%
  \makeatletter%
  \providecommand\color[2][]{%
    \errmessage{(Inkscape) Color is used for the text in Inkscape, but the package 'color.sty' is not loaded}%
    \renewcommand\color[2][]{}%
  }%
  \providecommand\transparent[1]{%
    \errmessage{(Inkscape) Transparency is used (non-zero) for the text in Inkscape, but the package 'transparent.sty' is not loaded}%
    \renewcommand\transparent[1]{}%
  }%
  \providecommand\rotatebox[2]{#2}%
  \newcommand*\fsize{\dimexpr\f@size pt\relax}%
  \newcommand*\lineheight[1]{\fontsize{\fsize}{#1\fsize}\selectfont}%
  \ifx\svgwidth\undefined%
    \setlength{\unitlength}{494.43591584bp}%
    \ifx\svgscale\undefined%
      \relax%
    \else%
      \setlength{\unitlength}{\unitlength * \real{\svgscale}}%
    \fi%
  \else%
    \setlength{\unitlength}{\svgwidth}%
  \fi%
  \global\let\svgwidth\undefined%
  \global\let\svgscale\undefined%
  \makeatother%
  \begin{picture}(0.33,0.10)%
    \lineheight{1}%
    \setlength{\unitlength}{7.5cm}
    \setlength\tabcolsep{0pt}%
    \put(0,0){\includegraphics[width=\unitlength,page=1]{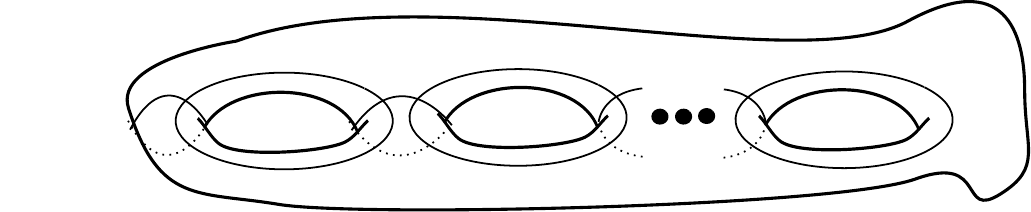}}%
    \put(0.23019068,0.14658527){\makebox(0,0)[lt]{\lineheight{1.25}\smash{\begin{tabular}[t]{l}$ \alpha_{2} $\end{tabular}}}}%
    \put(0.131492,0.11169476){\makebox(0,0)[lt]{\lineheight{1.25}\smash{\begin{tabular}[t]{l}$ \alpha_{1} $\end{tabular}}}}%
    \put(0.35505739,0.12712547){\makebox(0,0)[lt]{\lineheight{1.25}\smash{\begin{tabular}[t]{l}$ \alpha_{3} $\end{tabular}}}}%
    \put(0.46046434,0.14496365){\makebox(0,0)[lt]{\lineheight{1.25}\smash{\begin{tabular}[t]{l}$ \alpha_{4} $\end{tabular}}}}%
    \put(0.76209042,0.14172028){\makebox(0,0)[lt]{\lineheight{1.25}\smash{\begin{tabular}[t]{l}$ \alpha_{2p} $\end{tabular}}}}%
    \put(0,0){\includegraphics[width=\unitlength,page=2]{Figures/typeA2p.pdf}}%
    \put(0.93800925,0.10789499){\makebox(0,0)[lt]{\lineheight{1.25}\smash{\begin{tabular}[t]{l}$ b_1 $\end{tabular}}}}%
    \put(-0.00008295,0.10766572){\makebox(0,0)[lt]{\lineheight{1.25}\smash{\begin{tabular}[t]{l}$ A_{2p} $\end{tabular}}}}%
  \end{picture}%
\endgroup%

		\end{minipage}%
	}%
	\subfigure{
		\begin{minipage}[t]{0.5\linewidth}
			
			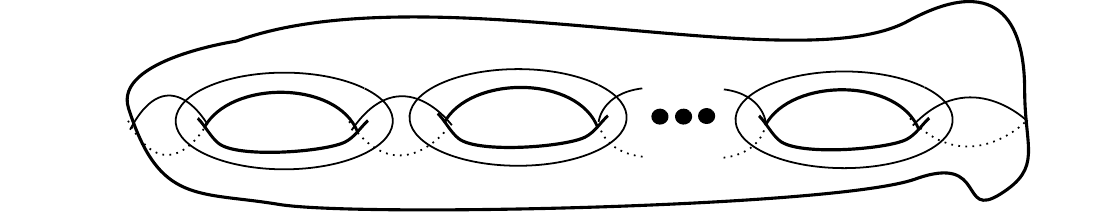
		\end{minipage}%
	}%
	
	\subfigure{
		\begin{minipage}[t]{0.5\linewidth}
			
			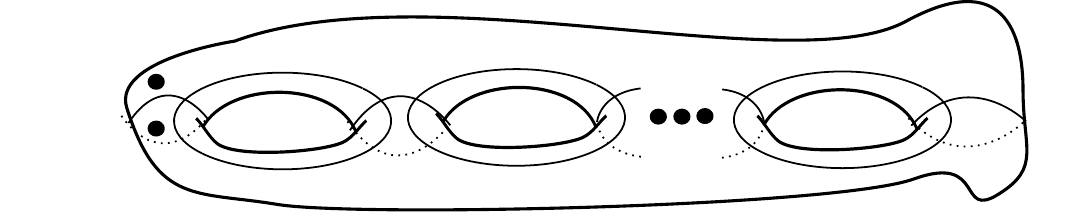
		\end{minipage}
	}%
	\subfigure{
		\begin{minipage}[t]{0.5\linewidth}
			
			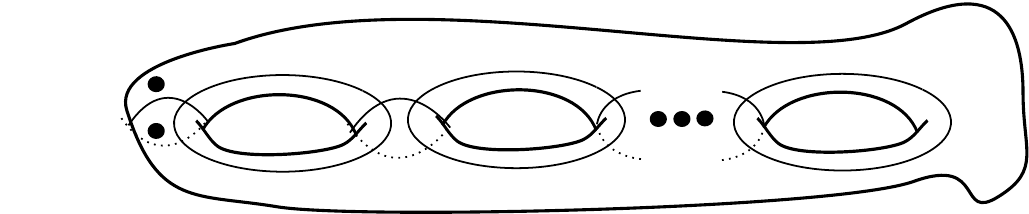
		\end{minipage}
	}%
	
	\subfigure{
		\begin{minipage}[t]{0.5\linewidth}
			
			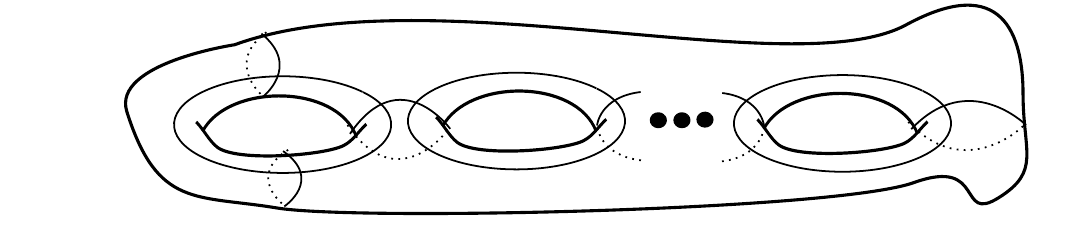
		\end{minipage}
	}%
	\subfigure{
		\begin{minipage}[t]{0.5\linewidth}
			
			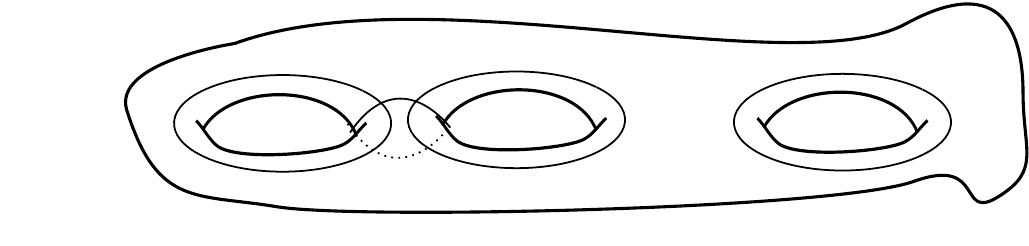
		\end{minipage}
	}%
	
	\subfigure{
		\begin{minipage}[t]{0.5\linewidth}
			
			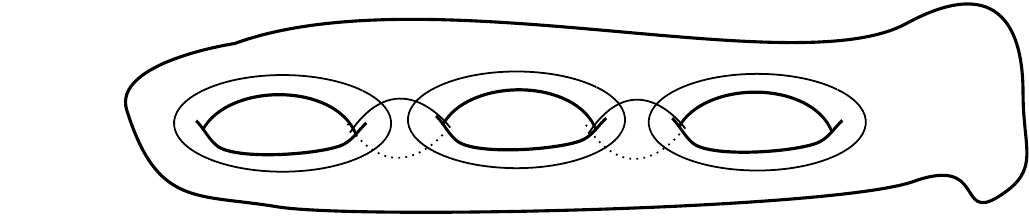
		\end{minipage}%
	}%
	\subfigure{
		\begin{minipage}[t]{0.5\linewidth}
			
			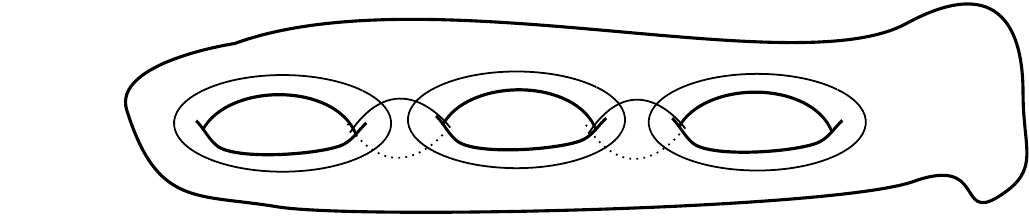
		\end{minipage}%
	}%
	\centering
	\caption{Perron-Vannier representations of type $ ABDE $}
	\label{PVrepABDEfig}
\end{figure}

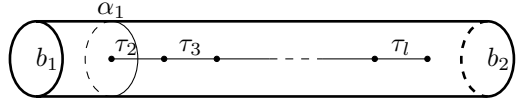
\begin{figure}[htbp]
	\begin{tikzpicture}[scale=1]
		\draw (-2,-0.5) arc (-90:90:0.35 and 0.5);
		
		\draw[line width=1] (-3,0) ellipse (0.35 and 0.5)
		(-3,0.5)--(3,0.5) (-3,-0.5)--(3,-0.5)
		(3,-0.5) arc (-90:90:0.35 and 0.5)
		;
		\draw[dashed, line width=1] (3,0.5) arc (90:270:0.35 and 0.5)
		;
		\draw[dashed](-2,0.5) arc (90:270:0.35 and 0.5);
		
		\filldraw  (-1.3,0) circle [radius=1pt]
		(-2,0) circle [radius=1pt]
		(-0.6,0) circle [radius=1pt]
		(1.5,0) circle [radius=1pt]
		(2.2,0) circle [radius=1pt]
		;
		\draw (-2,0)--(-1.3,0)
		(-1.3,0)--(-0.6,0)
		(-0.6,0)--(0.1,0)
		(0.8,0)--(1.5,0)
		(1.5,0)--(2.2,0)
		;
		\draw[dashed] (0.1,0)--(0.8,0);
		
		\node(a1) at (-2,0.65) {$ \alpha_{1} $};
		\node(t1) at (-1.8,0.15) {$ \tau_2 $};
		\node(t2) at (-0.95,0.15) {$ \tau_3 $};
		\node(t1) at (1.85,0.15) {$ \tau_l $};
		\node(b1) at (-2.85,0) {$ b_1 $};
		\node(b2) at (3.15,0) {$ b_2 $};

	\end{tikzpicture}
	\caption{Graph representations of type $ B_{l} $}	
	\label{GrepBlfig}
\end{figure}

 $ \text{PMod}(S_{g,1}, \mathcal{P}_{n}) $ admits a presentation in term of Artin group.

\begin{thm}\cite{LP}  \label{gr1pprethm}
	Let $ g \ge 1, n \ge 0 $ and $ \Gamma_{g,1,n} $ Coxeter graph as in Figure \ref{coxgrafig}. Then $ \text{\em PMod}(S_{g,1}, \mathcal{P}_{n}) $ is isomorphic to the quotient of $ A(\Gamma_{g,1,n}) $  with the following relations:  \\
	\textbullet \, Relations from $ \text{\em Mod}(S_{g,1})  $:
	\[ \Delta^{4}(y_{1},y_{2},y_{3},z)=\Delta^{2}(x_{0},y_{1},y_{2},y_{3},z), \,\,\,  g \ge 2 \]
	\[ \Delta^{2}(y_{1},y_{2},y_{3},y_{4},y_{5},z) =\Delta(x_{0},y_{1},y_{2},y_{3},y_{4},y_{5},z).  \,\,\, g \ge 3 \]
	\textbullet \, Relations of commutations:
	\[ x_{k}\Delta^{-1}(x_{i+1,x_{j},y_{1}})x_{i}\Delta(x_{i+1,x_{j},y_{1}})=\Delta^{-1}(x_{i+1,x_{j},y_{1}})x_{i}\Delta(x_{i+1,x_{j},y_{1}}) x_{k}, \,\,\,0\le k <j<i\le n-1  \]
	\[ y_{2}\Delta^{-1}(x_{i+1,x_{j},y_{1}})x_{i}\Delta(x_{i+1,x_{j},y_{1}})=\Delta^{-1}(x_{i+1,x_{j},y_{1}})x_{i}\Delta(x_{i+1,x_{j},y_{1}})y_{2}, \,\,\, 0 \le j < i \le n-1, g\ge 2  \]
	\textbullet \, Relations between fundamental elements:
	\[ \Delta(x_{0},x_{1},y_{1},y_{2},y_{3},z)=\Delta^{2}(x_{1},y_{1},y_{2},y_{3},z),\,\,\, g\ge 2 \]
	\[ \Delta(x_{i},x_{i+1},y_{1},y_{2},y_{3},z)\Delta^{-2}(x_{i+1},y_{1},y_{2},y_{3},z)=\Delta(x_{0},x_{i},x_{i+1},y_{1})\Delta^{-2}(x_{0},x_{i+1},y_{1}). 1 \le i \le n-1, g\ge 2 \]
	Here $ x_{i}, y_{j},z $ correspond to Dehn twists about $ \alpha_{i}, \beta_{j}, \gamma $, respectively, see Figure \ref{genb1fig}.
\end{thm}

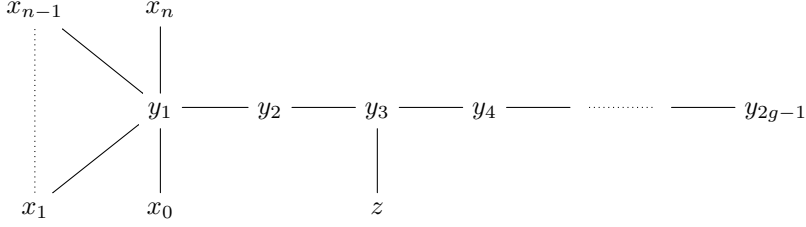
\begin{figure}[htbp]
	\[ 	\xymatrix{
		x_{n-1} \ar@{-}[dr] &  x_{n} \ar@{-}[d] &  & &  &   \\
		&  y_{1}\ar@{-}[r] & y_{2} \ar@{-}[r] & y_{3} \ar@{-}[r]& y_{4} \ar@{-}[r] & \ar@{.}[r] &  \ar@{-}[r] & y_{2g-1}\\
		x_{1} \ar@{.}[uu] \ar@{-}[ur]   & x_{0} \ar@{-}[u]&  &  z \ar@{-}[u]  \\
	} \]
	\caption{The Coxeter graph $ \Gamma_{g,1,n} $ associated with $ \text{PMod}(S_{g,1}, \mathcal{P}_{n}) $  }
	\label{coxgrafig}
\end{figure}

\begin{figure}[H]
	\centering
	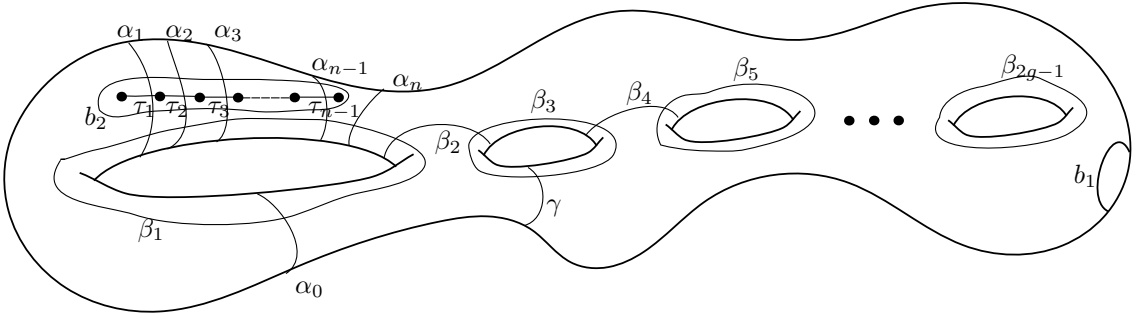
	\caption{Generators for  $ \text{PMod}(S_{g,1}, \mathcal{P}_{n}) $}
	\label{genb1fig}
\end{figure}

\section{Presentations of mapping class groups stabilizing boundaries}

In this section, we discuss mapping class group $\mathcal{MCG}(S,M)$ for genus $ g=0 $ and $ g \ge 1 $ based on presentations introduced in \cite{FM,LP}.

\subsection{Preliminaries} \quad

Let $ (S,M) $ be a bordered surface with marked points. According to the results of \cite{Gu,TBIS,BY}, fixing a triangulation of $ (S,M) $, two elements $ f,g \in \mathcal{MCG}(S,M) $ are different if and only if the images of the triangulation under $ f $ and $ g $ are different.

For a pair of boundary components with same marked points and an arc connecting them, let $ \zeta $ be a loop surrounding only these boundary components and the arc such that  $ \zeta $ cut out a area homotopic to the union of these  boundary components and the arc.

\begin{defi}
    (i)  a \textbf{ $ 1/m $ twist} $ T_{1/m} $ about a boundary component $ b $  with $ m $ marked points is defined as a homeomorphism of a regular neighborhood  of the boundary component mapping each marked point of the boundary component to the next marked point such that $  T_{1/m}^{m} $ is the Dehn twist along $ b $.

	(ii)	A \textbf{half twist} $ T_{h} $ about two  boundary components $ b_{1},b_{2} $ with $ n $ marked points respectively along an arc $ \gamma $ connecting them is defined as a homeomorphism satisfying:
	
	(a) the neighborhood containing the two  boundary components and $ T_{h} $ maps $ b_{1} $ to $ b_{2} $.
	
	(b) $  T_{h}$  maps one boundary component (or puncture) to another one such that $ T_{h}^{2} $ is the Dehn twist along $ \zeta $.
\end{defi}

\begin{figure}[htbp]
	\centering
	\subfigure{
		\begin{tikzpicture}[scale=1.25]
			\draw[dashed] (0,0)  circle(1)  (3,0)  circle(1) ;
			\draw (0,0) circle (0.5) (3,0)  circle(0.5);
			
			\filldraw  (-0.5,0) circle [radius=1pt]
			(0.5,0) circle [radius=1pt]
			(2.5,0) circle [radius=1pt]
			(3.5,0) circle [radius=1pt]
			(0,0.5) circle [radius=1pt]
			(0,-0.5) circle [radius=1pt]
			(3,0.5) circle [radius=1pt]
			(3,-0.5) circle [radius=1pt];
			
			\draw[dashed] (-1,0) -- (-0.5,0);

			\draw[->] (1.3,0) -- (1.7,0);
			
			\draw[dashed] (2,0) ..controls (2.3,-1) and (3,-0.8) .. (3,-0.5);

			\node(b) at (0,-0.2) {$ b$};
			\node(b) at (3,-0.2) {$ b$};
		\end{tikzpicture}
	}%
	\subfigure{
		\begin{tikzpicture}[scale=1]
			\draw[dashed] (0,0) circle(0.15) circle(1.25)  (3,0) circle(0.15) circle(1.25) ;
			\draw (0,0) circle (0.8);
			
			\draw (0.5,0)  circle (0.25)
			(-0.5,0) circle (0.25)
			(2.5,0) circle (0.25)
			(3.5,0) circle (0.25);
			
			\filldraw  (0.25,0) circle [radius=1pt]
			(0.75,0) circle [radius=1pt]
			(-0.25,0) circle [radius=1pt]
			(-0.75,0) circle [radius=1pt]
			
			(2.25,0) circle [radius=1pt]
			(2.75,0) circle [radius=1pt]
			(3.25,0) circle [radius=1pt]
			(3.75,0) circle [radius=1pt]
			
			(0.5,0.25) circle [radius=1pt]
			(0.5,-0.25) circle [radius=1pt]
			(-0.5,0.25) circle [radius=1pt]
			(-0.5,-0.25) circle [radius=1pt]
			
			(2.5,0.25) circle [radius=1pt]
			(2.5,-0.25) circle [radius=1pt]
			(3.5,0.25) circle [radius=1pt]
			(3.5,-0.25) circle [radius=1pt];

			\draw[dashed] (-1.25,0) -- (-0.75,0)  (1.25,0) -- (0.75,0);
			\draw (0.25,0) -- (-0.25,0) (2.75,0) -- (3.25,0);
			
			\draw[->] (1.35,0) -- (1.65,0);
			
			\draw[dashed] (1.75,0) ..controls (2.8,-1.2) and (4.5,-0.3) .. (3.75,0)
			(4.25,0) ..controls (3.2,1.2) and (1.5,0.3) .. (2.25,0);
			
			\node(ze) at (0,1) {$ \zeta$};
			\node(gm) at (0,-0.3) {$ \gamma$};
			\node(gm) at (3,-0.3) {$ \gamma$};
		\end{tikzpicture}
	}
	\caption{$1/4$ twist and half twist about the boundary component}
	\label{nhtb}
	
\end{figure}
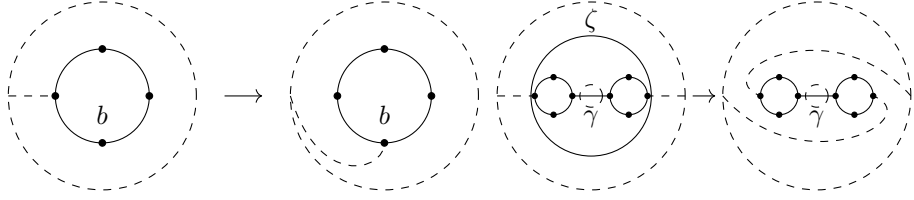

Firstly, we consider the mapping class group on the area as in Figure \ref{bdarc}.

\begin{lem}\label{restrictlem}
The mapping class group of the annulus $ (S,M) $ with one marked point on the outer boundary component $b_{k}' $ and  $ m $ marked points on another boundary component  $ b_{k} $  as in Figure \ref{restri}  is  the  cyclic group generated by a $ 1/m $ twist about $ b_{k} $.
\end{lem}
\begin{proof}
	According to \cite{ASS}, $ \mathcal{MCG}(S,M) \cong \text{\em Aut}^{+}\mathcal{A}(S,M) $ is isomorphic to $ \mathbb{Z} $. One can check the generator is the $ 1/m $ twist about $ b_{k} $.
	
\end{proof}

\begin{figure}[htbp]
	\centering
		\begin{tikzpicture}[scale=0.7]
			\filldraw[fill=gray!10] (1,0) arc [start angle=0, end angle=270, radius=1];
			\filldraw[dotted,fill=gray!10] (0,-1) arc [start angle=270, end angle=360, radius=1];
			\filldraw
			(0,1) circle [radius=2pt]
			(0,3) circle [radius=2pt];
			
			\filldraw  (1,0) circle [radius=2pt]
			(-1,0) circle [radius=2pt]
			(0,-1) circle [radius=2pt];
			
			\draw (0,0) circle (3);

			\draw (0,1) ..controls (2,2) and (1.5,-1.5) .. (0,-1.5)
			(0,1) ..controls (-2,2) and (-1.5,-1.5) .. (0,-1.5);
			
			\draw (0,3) ..controls (2,2) and (2.5,-2.5)  .. (0,-2.5)
			(0,3) ..controls (-2,2) and (-2.5,-2.5) .. (0,-2.5);
			
			\draw (0,3) ..controls (1.75,1.75) and (2,-2)  .. (0,-2)
			(0,1) ..controls (-2.25,2.25) and (-2,-2) .. (0,-2);
			\draw(0,3) to (0,1);
			
			\node(rk) at (0,-1.7) {$ \gamma_{k} $};
			\node(zek) at (1.8,-1.6) {$ \zeta_{k} $};
			\node(p) at (-0.1,2.7) {$ p $};
			\node(b1) at (0.3,2) {$ \beta_{k_{1}} $};
			\node(b2) at (0,-2.2) {$ \beta_{k_{2}} $};
			
			\node(bk) at (-0.7,0) {$b_{k} $};
			\node(bk') at (-2.7,0) {$b_{k}' $};

		\end{tikzpicture}
	\caption{The triangulation of the annulus in Lemma \ref{restrictlem}}
	\label{restri}

\end{figure}
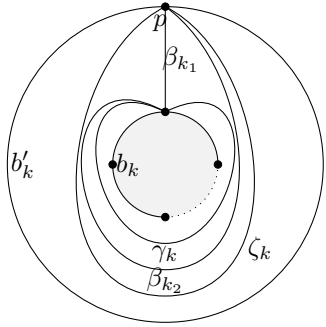

Let $ (S=S_{g,r}, \mathcal{B}, \mathcal{P}_{n},M) $ be the surface of genus $ g $, with the set $ \mathcal{B} $ of boundary components, the set $ \mathcal{P}_{n} $ of punctures and the set $ M$ of marked points.

  Define the quotient surface $ \widetilde{S}  =S / \mathcal{B}$ obtained from $S$ over    the following equivalence  relation  $\sim$ such that for $x, y \in S$, $x \sim y$ if and only if  $ x $ and $ y $ are the same or in the same boundary component.

   In  $ \widetilde{S}$, each boundary component in $ \mathcal{B} $ is degenerated to a puncture. We use $\mathcal{B}_0$ to denote the set of these punctures in  $ \widetilde{S}$. Then $ \widetilde{S}= S_{g,0} $ is the closed surface of genus $ g $, with the set of punctures equal to  $ \mathcal{P}_{n}\cup \mathcal{B}_0 $.

For $ f\in \mathcal{MCG}(S,M) $, suppose $ f $ is represented by a homeomorphism $ \phi $, then $ \phi $  maps boundary components to boundary components. Thus we get the
homeomorphism $ \psi $ of $  \widetilde{S} $ induced by $ \phi $ satisfying  the following commutative diagram:
	\[ 	\xymatrix{
		S \ar[r]^{\phi} \ar[d]^{p} & S \ar[d]^{p}\\
		\widetilde{S} \ar[r]^{\psi} & \widetilde{S} \\
	} \]
where $ p $ is the canonical projection.
 Besides if $ \phi' $ is homotopic to $ \phi $, then two homeomorphisms induced by $ \phi $ and $ \phi' $ are homotopic to each other. Thus we have a homomorphism $ \pi $ from $ \mathcal{MCG}(S,M) $ to $ \text{Mod}(\widetilde{S},\mathcal{P}_{n}\cup\mathcal{B}_{0}) $, such that $ \pi(\phi)= \psi $. Then we have the following exact sequence:
\begin{equation}\label{quosq}
	1 \to \text{Ker}\pi \to \mathcal{MCG}(S,M) \to \text{Im}\pi \to 1.
\end{equation}

Thus we can get a presentation of $ \mathcal{MCG}(S,M)  $ via $ \text{Ker}\pi $ and $  \text{Im}\pi  $.

Next, we calculate $ \text{Ker}\pi $.

\begin{lem}\label{kerlem}
For the case  $(S,\mathcal{P}_{n}) \neq( S_{0,2}, \emptyset) $, $ \text{\em Ker}\pi $ is a free abelian group generated by $ \{T_{b_{k}}, b_{k} \in \mathcal{B} \} $, where $ T_{b_{k}} $ is the $ 1/m $ twist about $ b_{k} $.
\end{lem}

\begin{proof}
	Fixing a triangulation $ \mathcal{T} $ of $ (S,M) $ such that for each boundary component $ b_{k} $, the triangulation $ \mathcal{T} $ contains the arc $ \gamma_{k} $ showed in Figure \ref{bdarc}, where $ \zeta_{k} $ is a loop surrounding only $ b_{k}  $, $  f\in \mathcal{MCG}(S,M) $ is uniquely determined by the triangulation $ f(\mathcal{T}) $. In particular, $ f  $ is the identity if and only if $ \gamma $ is homotopic to $ f(\gamma) $ for each $ \gamma \in \mathcal{T} $.
		
	If $ r<2 $, the element in $ \text{Im}\pi $ fix boundary components.
	
	If $ r \ge 2 $ and $S =S_{0,2} $ and $ \mathcal{P}_{n} =\emptyset $,
	Suppose $ f\in \mathcal{MCG}(S,M) $ maps $ b $ to $ b' $, we choose the fixed triangulation $ \mathcal{T} $ containing a arc $ \gamma $ connected to $ b $ but not to $ b' $, then $ \widetilde{\gamma} $ induced by $ \gamma $ in $ \widetilde{S} $ is not homotopic to $ \pi(f)(\widetilde{\gamma} ) $, and hence $ \pi(f) \notin \text{Im}\pi $. Thus the element in $ \text{Im}\pi $ fix boundary components.
	
	Note that if $ \gamma \in \mathcal{T} $ does not intersect with $ \zeta_{k} $, then for $  f\in \mathcal{MCG}(S,M) $, the $ \widetilde{\gamma} $ induced by $ \gamma $ in $ \widetilde{S} $ is homotopic to $ \pi(f)(\widetilde{\gamma} ) $ if and only $ \gamma  $ is homotopic to $f(\gamma)  $. Thus if $f\in \text{Ker}\pi  $, $ f(\gamma) $ is homotopic to $ \gamma $, and we only need to consider the arcs of $ \mathcal{T} $ lying in the areas bounded by $ \zeta_{k} $.
	
	For the case  $S \neq S_{0,2} $ or $ \mathcal{P}_{n} \neq \emptyset $, $ f\in \text{Ker}\pi $ fix boundary components,  we only need to consider the case that arcs of $ \mathcal{T} $ and their images lying in the same connected components.
	
	It is equivalent to consider the triangulations in Figure \ref{restri}  and their images under $ f \in \mathcal{MCG}(S,M)  $ which fix boundary components.
	Thus $ f $ induces a mapping class on the area bounded by $ \zeta_{k}$ and two mapping class  on different area commutate are commutative.
	Then we have a homomorphism:
	\[ \varphi : \;\; \text{Ker}\pi \to \prod_{b_{k} \in \mathcal{B}}\mathcal{MCG}(S_{b_{k}},M_{b_{k}}),   \]
	where $ (S_{b_{k}},M_{b_{k}}) $ is the area bounded by $ \zeta_{k} $. Clearly, $ \varphi $ is epimorphic.
	
	Since any arc in $ \mathcal{T} $ not intersect with $ \zeta_{k} $ is fixed under $ f \in \text{Ker}\pi $, $ \varphi $ is monomorphic. Thus $ \text{Ker}\pi \cong \prod_{b_{k} \in \mathcal{B}}\mathcal{MCG}(S_{b_{k}},M_{b_{k}}) $ and according to Lemma \ref{restrictlem}, $ \mathcal{MCG}(S_{b_{k}},M_{b_{k}}) $ is a free abelian group generated by $ \{T_{b_{k}}, b_{k} \in \mathcal{B} \} $, where $ T_{b_{k}} $ is the $ 1/m $ twist about $ b_{k} $.
	
\end{proof}

	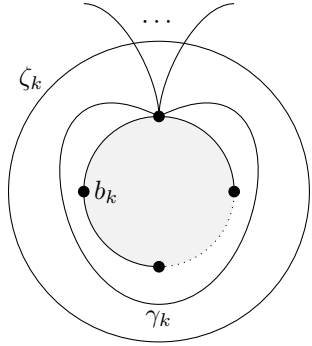
\begin{figure}[htbp]
	\begin{tikzpicture}[scale=1]
		\filldraw[fill=gray!10] (1,0) arc [start angle=0, end angle=270, radius=1];
		\filldraw[dotted,fill=gray!10] (0,-1) arc [start angle=270, end angle=360, radius=1];
		\filldraw
		(0,1) circle [radius=2pt];
		
		\filldraw  (1,0) circle [radius=2pt]
		(-1,0) circle [radius=2pt]
		(0,-1) circle [radius=2pt];
		
		\draw (0,0) circle (2);
		
		\draw (0,1) ..controls (0,1.5) and (0.5,2.5)  .. (1,2.5)
		(0,1) ..controls (0,1.5) and (-0.5,2.5)  .. (-1,2.5);
		
		\draw (0,1) ..controls (2,2) and (1.5,-1.5) .. (0,-1.5)
		(0,1) ..controls (-2,2) and (-1.5,-1.5) .. (0,-1.5);
		
		\node(rk) at (0,-1.7) {$ \gamma_{k} $};
		\node(bk) at (-0.7,-0) {$b_{k} $};

		\node(zek) at (-1.7,1.5) {$\zeta_k$};
		\node(dots) at (0,2.25) {$ \cdots $};
	\end{tikzpicture}
	\caption{The area bounded by $ \zeta_{k}$}	
	\label{bdarc}
\end{figure}

Before giving the presentation, we need the following basic lemma for Dehn twists and groups.

\begin{lem}\cite{FM} \label{basicdt}
	 For any $ f \in \text{\em Mod}(S,M)  $ and any two loops $ \zeta_{1}, \zeta_{2} $ in $ S $, we have
	
	(i)  $ T_{f(\zeta_{1})}=fT_{\zeta_{1}}f^{-1} $;
	
	(ii) $ f $ commutes with $ T_{\zeta_{1}} \Leftrightarrow f(\zeta_{1})=\zeta_{1}  $;
	
	(iii) The minimal intersection number of $\zeta_{1}  $ and $ \zeta_{2} $ $i(\zeta_{1},\zeta_{2}) =0$ $ \Leftrightarrow $ $ T_{\zeta_{1}}T_{\zeta_{2}}=T_{\zeta_{2}}T_{\zeta_{1}} $.
\end{lem}

\begin{lem} \cite{FM} \label{lanternlem}
	Let $ \alpha_{1}, \alpha_{2}, \alpha_{3}, \beta_{1}, \beta_{2}, \beta_{3}, \beta_{4} $  be loops in a surface $S$ form a subsurface as in Figure \ref{lanternfig}, then we have
	\[ T_{\alpha_{1}}T_{\alpha_{2}}T_{\alpha_{3}}=T_{\beta_{1}}T_{\beta_{2}}T_{\beta_{3}}T_{\beta_{3}}. \]
\end{lem}

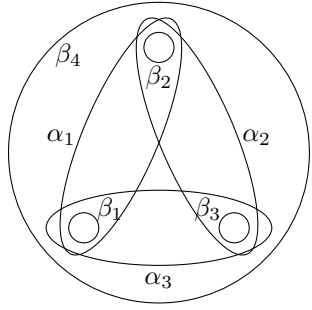
\begin{figure}[htbp]
	\begin{tikzpicture}[scale=1]
		\draw (0,0) circle (2)  (0,1.4) circle (0.2) (-1,-1) circle (0.2) (1,-1) circle (0.2);
	
		\draw (0,-1) ellipse (1.5 and 0.5);
	
		\draw[rotate=247] (0,-0.55) ellipse (1.7 and 0.5);
		\draw[rotate=113] (0,-0.55) ellipse (1.7 and 0.5);
		
		\node(b2) at (0,1) {$\beta_2$};
		\node(b1) at (-0.65,-0.75) {$\beta_1$};
		\node(b3) at (0.65,-0.75) {$\beta_3$};
		\node(a1) at (-1.3,0.2) {$\alpha_1$};
		\node(a2) at (1.3,0.2) {$\alpha_2$};
		\node(a3) at (0,-1.7) {$\alpha_3$};
		\node(a3) at (-1.2,1.3) {$\beta_4$};
	\end{tikzpicture}
	\caption{The relation in Lemma \ref{lanternlem}}	
	\label{lanternfig}
\end{figure}

Suppose we have the following exact sequence of groups:
\[  1 \to K \to G \xrightarrow{\rho} H \to 1, \]
$ K $ admits a presentation $\langle S_{K} | R_{K} \rangle  $ and $ H $ admits a presentation $\langle S_{H} | R_{H} \rangle  $.

Choose a subset
\[ \tilde{S}_{H}=\{\tilde{x}| x\in S_{H} \} \]
 of $ G $ such that $ \rho(\tilde{x})=x $. Let $ r= x_{1}^{\varepsilon_{1}}\cdots x_{n}^{\varepsilon_{n}} \in R_{H} $, we write $ \tilde{r}=\tilde{x}_{1}^{\varepsilon_{1}} \cdots \tilde{x}_{n}^{\varepsilon_{n}}$. Since $ \rho(\tilde{r})=r=1 $ in $ H $, $ \tilde{r} \in \text{Ker}\rho $ and we can choose $ w_{r} \in S_{K}$ such that $  w_{r}  $ and $ \tilde{r} $ represent the same element of $ G $. We write
\[ R_{1}=\{\tilde{r}w_{r}^{-1}| r\in R_{H}. \}  \].
Let $\tilde{x} \in \tilde{S}_{H}  $ and $ y\in S_{K} $. Since $ K $ is normal subgroup in $ G $, there is $ v(x,y) \in S_{K}  $ such that both $ v(x,y) $ and $ \tilde{x}y\tilde{x}^{-1} $ represent the same element of $ G $. Set
\[ \tilde{x}y\tilde{x}^{-1}v(x,y)^{-1} |  \tilde{x} \in \tilde{S}_{H}, y \in S_{K}. \]

Then we have the following lemma.
\begin{lem} \cite{LP} \label{basicgp}
$ G $ admits a presentation:
\[ G= \langle  \tilde{S}_{H} \cup S_{K} | R_{1}\cup R_{2}\cup R_{K}  \rangle. \]
\end{lem}

Suppose $S=S_{g,r} $ has $ n_{k} $ boundary components with $ k $ marked points, the set $ \mathcal{B}=\{b_{1},\cdots,b_{r}\} $ of $r$ boundary components with $k_i$ mark points on the boundary component $b_i$ for $i\in [1,r]$ satisfying $k_1\leqslant\cdots\leqslant k_r$, the set $ \mathcal{P}_{n}=\{p_1,\cdots,p_n\} $ of punctures and the set $ M $ of marked points. We give an order $\prec$ on the set $ \mathcal{P}_{n} \cup \mathcal{B}   $ satisfying
\begin{equation}\label{order}
p_1\prec\cdots\prec p_n\prec b_{1}\prec\cdots\prec b_{r}.
 \end{equation}
 Then we can give a map
 \begin{equation}\label{omega}
  \omega : \mathcal{P}_{n} \cup \mathcal{B}  \to  \mathcal{P}_{n+r}=\mathcal{P}_{n} \cup \mathcal{B}_{0},
  \end{equation}
   keeping the order $\prec$.

In (\ref{quosq}), $ \text{Im}\pi $ is the subgroup of  $\text{Mod}(\widetilde{S},\mathcal{P}_{n+r})  $ consisting of elements map each puncture to another puncture and map each boundary component to another  boundary component with the same number of marked points. Applying (\ref{pureexsq}), $ \text{Im}\pi =\theta^{-1}(\Sigma_{n}\times \prod_{n_{k}\neq 0}\Sigma_{n_{k}}) $, and we have the following exact sequence:
\begin{equation}\label{imsq}
	1 \to \text{PMod}(\widetilde{S},\mathcal{P}_{n+r}) \to \text{Im}\pi \xrightarrow{\theta} \Sigma_{n}\times \prod_{n_{k}\neq 0}\Sigma_{n_{k}}\to 1.
\end{equation}		
		
Note that $\Sigma_{S}=\Sigma_{n}\times \prod_{n_{k}\neq 0}\Sigma_{n_{k}} $	admits a presentation:
\begin{equation}\label{permupresentation}
	\begin{array}{lll}
		\Sigma_{S}= \langle \text{Per}_{i}, i \in I | & \text{Per}_{i}\text{Per}_{j}=\text{Per}_{j}\text{Per}_{i} & |i-j|>1 \\
		& \text{Per}_{i}\text{Per}_{j}\text{Per}_{i}=\text{Per}_{j}\text{Per}_{i}\text{Per}_{j} &  |i-j|=1\\
		& \text{Per}_{i}^{2}=1 \rangle, & \\
	\end{array}
\end{equation}	
where $ \text{Per}_{i} =(i, i+1) $, $ I = [n+r-1]\backslash \{n, n+\Sigma_{i=1}^{j}n_{i} | j\in \mathbb{N}  \}  $.

\subsection{Presentations of mapping class group stabilizing boundaries for genus $ g=0 $} \quad

Now  we calculate $ \text{Im}\pi $  in (\ref{imsq}) for genus $ g=0 $.

According to \cite{FM}, there is a epimorphism $ \varphi : \text{ Mod}(S_{0,1}, \mathcal{P}_{n}) \to \text{ Mod}(S_{0,0}, \mathcal{P}_{n})$ induced by the inclusion of $ S_{0,1} $ to $  S_{0,0}  $, and the kernel is generated by the following relations:
\begin{equation}\label{kernelphig0}
	\begin{array}{c}
		 (\sigma_{1}\cdots\sigma_{n-1})^{n}=1, \\
		  \sigma_{1}\cdots \sigma_{n-1}\sigma_{n-1} \cdots \sigma_{1}=1,
	\end{array}
\end{equation}
where $\sigma_{i}$ are generators of the presentation (\ref{braidpre}) of $ \text{ Mod}(S_{0,1})=B_{n} $.

According to the presentation (\ref{pbn}) of $ \text{P}B_{n}$ and the presentation of $ \text{ Mod}(S_{0,0},\mathcal{P}_{n}) $ in Theorem \ref{S0rpresentation},
since
\begin{equation*}
	(\sigma_{1}\cdots\sigma_{n-1})^{n}=(a_{12}\cdots a_{1n})\cdots(a_{n-2n-1}a_{n-2n})(a_{n-1n} ),
\end{equation*}
$$ \sigma_{1}\cdots \sigma_{n-1}\sigma_{n-1} \cdots \sigma_{1}=a_{12} \cdots a_{1n-1}a_{1n},
$$
 $ \text{ PMod}(S_{0,0},\mathcal{P}_{n}) $ admits a presentation as follows:
 \begin{equation}\label{PS0presentation}
 	\begin{array}{lll}
 		\text{ PMod}(S_{0,0},\mathcal{P}_{n}) = \langle a_{ij}\;  \forall i,j\in[1,n]\;\;  | & [a_{pq},a_{rs}]=1,  [a_{ps},a_{qr}]=1, & p <q <r < s\\
 		&a_{pr}a_{qr}a_{pq}=a_{qr}a_{pq}a_{pr}=a_{pq}a_{pr}a_{qr}, &  p <q <r \\
 		&    [a_{rs}a_{pr}a_{rs}^{-1}, a_{qs}]=1,              &  p <q <r < s  \\
 		& (a_{12}\cdots a_{1n})\cdots(a_{n-2n-1}a_{n-2n})(a_{n-1n})=1 &\\
 		&a_{12} \cdots a_{1n-1}a_{1n}=1 \rangle , &
 	\end{array}
 \end{equation}

where $ a_{ij} $ is the Dehn twist along $ \alpha_{ij} $, see Figure \ref{pS0generatorfig} and $ [a,b]=aba^{-1}b^{-1} $.

\begin{figure}[htbp]
	\begin{tikzpicture}[scale=1]
		\draw (0,0) circle (2);
		\draw[dashed] (0,0) ellipse (2 and 1);
		
		\filldraw  (-1,0.5) circle [radius=1pt]
		(-0.3,0.5) circle [radius=1pt]
		(0.3,0.5) circle [radius=1pt]
		(1,0.5) circle [radius=1pt];
		
		\draw[dotted] (0.4,0.5) -- (0.9,0.5) (-0.9,0.5) -- (-0.4,0.5) (-0.2,0.5) -- (0.2,0.5) ;
		
		\draw (-0.3,0.7) ..controls (-0.1,0.8) and (-0.2,0.1)  .. (0,0.2)
		(-0.3,0.7) ..controls (-0.6,0.8) and (-0.3,-0.1)  .. (0,0)
		(0.3,0.7) ..controls (0.1,0.8) and (0.2,0.1)  .. (0,0.2)
		(0.3,0.7) ..controls (0.6,0.8) and (0.3,-0.1)  .. (0,0);

		\node(p1) at (-1,0.58) {$ p_{1} $};
		\node(pi) at (-0.3,0.58) {$ p_{i} $};
		\node(pj) at (0.3,0.58) {$ p_{j} $};
		\node(pn) at (1,0.58) {$ p_{n} $};
		\node(aij) at (0,-0.13) {$ \alpha_{ij} $};
		
	\end{tikzpicture}
	\caption{Generators of $ \text{ PMod}(S_{0,0},\mathcal{P}_{n}) $}	
	\label{pS0generatorfig}
\end{figure}
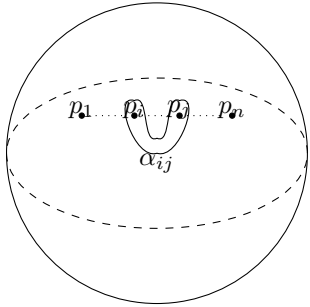

\begin{lem} \label{Impresentation}
Following notations given above, for the case  $(S,\mathcal{P}_{n}) \neq( S_{0,2}, \emptyset) $ and genus $ g=0 $, $ \text{\em Im}\pi $ is generated by $ \{\sigma_{k}, a_{ij}  |  k \in I; 1\le i \le n+r\} $ with the following relations. \\
\textbullet \, Relations from $ \Sigma_{S} $:
\begin{equation*}
	\begin{array}{ll}		
		  \sigma_{i}\sigma_{j}=\sigma_{j}\sigma_{i} & |i-j|>1,\\
		 \sigma_{i}\sigma_{i+1}\sigma_{i}=\sigma_{i+1}\sigma_{i}\sigma_{i+1}, &\\
		 \sigma_{i}^{2}=a_{i-1i}. & \\
	\end{array}
\end{equation*} 	
\textbullet \, Relations from the presentation of $ \text{\em PMod}(S_{0,0},\mathcal{P}_{n})  $ in (\ref{PS0presentation}):
\begin{equation*}
	\begin{array}{ll}		
	 [a_{pq},a_{rs}]=1,  [a_{ps},a_{qr}]=1, & p <q <r < s\\
     a_{pr}a_{qr}a_{pq}=a_{qr}a_{pq}a_{pr}=a_{pq}a_{pr}a_{qr}, &  p <q <r \\
     \lbrack  a_{rs}a_{pr}a_{rs}^{-1}, a_{qs}  \rbrack =1.             &  p <q <r < s \\
	\end{array}
\end{equation*}
\textbullet \, Relations from the kernel of $ \varphi $ in (\ref{kernelphig0}):
\begin{equation*}
	\begin{array}{ll}		
		(a_{12}\cdots a_{1n})\cdots(a_{n-2n-1}a_{n-2}n)(a_{n-1n})=1, &\\
		a_{12} \cdots a_{1n-1}a_{1n}=1. & \\
	\end{array}
\end{equation*}
\textbullet \, Relations of commutations:
\begin{equation*}
	\begin{array}{ll}		
	 a_{ij}\sigma_{k}=\sigma_{k}a_{ij},  & k \neq i, i-1, j, j-1   \\
	\sigma_{k}^{-1}a_{ij}\sigma_{k}=	a_{i-1j}, & k=i-1 \\
	\sigma_{k}^{-1}a_{ij}\sigma_{k}= a_{ii+1}a_{i+1j}a_{ii+1}^{-1}, & k=i \\
	\sigma_{k}^{-1}a_{ij}\sigma_{k}=a_{ij-1}, & k=j-1\\
	\sigma_{k}^{-1}a_{ij}\sigma_{k}= a_{jj+1}a_{ij+1}a_{jj+1}^{-1}. & k=j \\
	\end{array}
\end{equation*}
\end{lem}

\begin{proof}
	Since genus $ g=0 $, $  \text{Mod}(\widetilde{S},\mathcal{P}_{n+r})  $ admits a presentation as in Theorem \ref{S0rpresentation} and for each $ \text{Per}_{i} \in \Sigma_{S} $, there is a lift $ \sigma_{i} $ of $ \text{Per}_{i} $ in $ \text{Mod}(\widetilde{S},\mathcal{P}_{n+r})  $. Besides, $ \text{PMod}(\widetilde{S},\mathcal{P}_{n+r}) $ admits a presentation as in (\ref{PS0presentation}). According to Lemma \ref{basicdt}, $ a_{ij}\sigma_{k}=\sigma_{k}a_{ij} $ if $ k \neq i, i-1, j, j-1  $,
	\[  \sigma_{k}^{-1}a_{ij}\sigma_{k}= \left\{ \begin{array}{ll}
		a_{i-1j}; & k=i-1 \\
		a_{ii+1}a_{i+1j}a_{ii+1}^{-1}; & k=i \\
		a_{ij-1}; & k=j-1\\
		a_{jj+1}a_{ij+1}a_{jj+1}^{-1}. & k=j\\
	\end{array}  \right.  \]
	And $ \sigma_{k}^{2}= a_{k-1k}$. Besides, one can check that the lifts of other relations in $\Sigma_{S}$ are equal to $ 1 $.
	Thus applying Lemma \ref{basicgp} to (\ref{imsq}) for genus $ g=0 $, we obtain the presentation of $ \text{Im}\pi$ as in Lemma \ref{Impresentation}.
	
\end{proof}

%
%

Next, we calculate $ \mathcal{MCG}(S,M)$ for genus $ g=0 $ based on (\ref{quosq}).

Suppose $S=S_{g,r} $ has the set $ \mathcal{B} $ of boundary components, the set $ \mathcal{P}_{n} $ of punctures and the set $ M $ of marked points.
For the fixed triangulation $ \mathcal{T} $ of $ (S,M) $ containing $ \zeta_{k} $ as in Figure \ref{bdarc} for each $ b_{k} $, we suppose $ \mathcal{T} $ contains the arcs $ \tilde{\gamma}_{i}, i \in [n+r-1] $ such that their images in quotient space $\widetilde{S} $ are arcs $ \gamma_{i} $ as in Figure \ref{sphht}.\\
There is a set of lifts of generators of $ \text{Im}\pi  $ in $ \mathcal{MCG}(S,M)$:

The lifts $\tilde{\alpha}_{ij}  $ of $ \alpha_{ij} $ as in Figure \ref{pS0generatorfig} are loops in $(S,M)  $, which are unique up to homotopy, the lifts $ \tilde{a}_{ij} $ of $ a_{ij} $ are defined to be Dehn twists along $ \tilde{\alpha}_{ij} $;

If $ k $ stand for a boundary component $ b $ with $ n_{b} $ marked points, then the lift $ \tilde{\sigma}_{k} $ is defined to be $\sigma_{k} $ which is a half twist about boundary components along $\tilde{\gamma}_{k}  $, and if $ k $ stand for punctures, the lift $ \tilde{\sigma}_{k}  $ of $ \sigma_{k}  $ is defined to be itself.

Let $ T_{b_{k}} , b_{k} \in \mathcal{B} $ be the $ 1/m $ twist about $ b_{k} $.

\begin{thm}\label{mcgSg0prop}
	Suppose $S=S_{g,r} $ has $ n_{k} $ boundary components with $ k $ marked points, the set $ \mathcal{B} $ of boundary components , the set $ \mathcal{P}_{n} $ of punctures, the set $ M $ of marked points.
	Let $ \tilde{\sigma}_{k} $, $ \tilde{a}_{ij} $ be lifts of $ \sigma_{k} $ and $ a_{ij} $ as in the presentation of $ \text{\em Im}\pi $ in Lemma \ref{Impresentation}, respectively, $ T_{b_{l}}  $ the $ 1/m $ twist about the boundary component $ b_{l} $, and $ I = [n+r-1]\backslash \{n, n+\Sigma_{i=1}^{j}n_{i} | j\in \mathbb{N}  \}  $.
	
	Then for the case  $ (S,\mathcal{P}_{n}) \neq( S_{0,2}, \emptyset)  $, if the genus $ g=0 $, then $  \mathcal{MCG}(S,M)  $ is generated by $ \{ \tilde{\sigma}_{k}, \tilde{a}_{ij}, T_{b_{l}} |  k \in I; 1\le i \le n+r; b_{l} \in \mathcal{B} \} $ with the following relations:\\
	\textbullet \, Relations from the presentation of $ \text{\em Im}\pi $:
	\begin{equation}\label{g0relation1}
		\begin{array}{ll}		
  \tilde{\sigma}_{i}\tilde{\sigma}_{j}=\tilde{\sigma}_{j}\tilde{\sigma}_{i} & |i-j|>1,\\
  \tilde{\sigma}_{i}\tilde{\sigma}_{i+1}\tilde{\sigma}_{i}=\tilde{\sigma}_{i+1}\tilde{\sigma}_{i}\tilde{\sigma}_{i+1}, &\\
 \tilde{\sigma}_{i}^{2}=\tilde{a}_{i-1i}, & \\
 \lbrack  \tilde{a}_{pq},\tilde{a}_{rs} \rbrack =1,  [\tilde{a}_{ps},\tilde{a}_{qr}]=1, & p <q <r < s\\
\tilde{a}_{pr}\tilde{a}_{qr}\tilde{a}_{pq}=\tilde{a}_{qr}\tilde{a}_{pq}\tilde{a}_{pr}=\tilde{a}_{pq}\tilde{a}_{pr}\tilde{a}_{qr}, &  p <q <r \\
    \lbrack \tilde{a}_{rs}\tilde{a}_{pr}\tilde{a}_{rs}^{-1}, \tilde{a}_{qs} \rbrack =1,              &  p <q <r < s \\
 (\tilde{a}_{12}\cdots \tilde{a}_{1n})\cdots(\tilde{a}_{n-2n-1}\tilde{a}_{n-2}n)(\tilde{a}_{n-1n})=1 &\\
\tilde{a}_{12}\cdots \tilde{a}_{1n-1}\tilde{a}_{1n} =1 & \\
 \tilde{a}_{ij}\tilde{\sigma}_{k}=\tilde{\sigma}_{k}\tilde{a}_{ij},  & k \neq i, i-1, j, j-1   \\
\tilde{\sigma_{k}}^{-1}\tilde{a}_{ij}\tilde{\sigma}_{k}=	\tilde{a}_{i-1j}, & k=i-1 \\
\tilde{\sigma_{k}}^{-1}\tilde{a}_{ij}\tilde{\sigma}_{k}= \tilde{a}_{ii+1}\tilde{a}_{i+1j}\tilde{a}_{ii+1}^{-1}, & k=i \\
\tilde{\sigma_{k}}^{-1}\tilde{a}_{ij}\tilde{\sigma}_{k}=\tilde{a}_{ij-1}, & k=j-1\\
\tilde{\sigma_{k}}^{-1}\tilde{a}_{ij}\tilde{\sigma}_{k}= \tilde{a}_{jj+1}\tilde{a}_{ij+1}\tilde{a}_{jj+1}^{-1}. & k=j \\
		\end{array}
	\end{equation}
\textbullet \, Relations of commutations:
\begin{equation}\label{g0relation2}
	\begin{array}{ll}		
	 T_{b_{i}}T_{b_{j}}=T_{b_{j}}T_{b_{i}}, & \\
	 \tilde{a}_{ij}T_{b_{k}}=T_{b_{k}}\tilde{a}_{ij},& \\
	 \tilde{\sigma}_{k}  T_{b_{k_{1}}} \tilde{\sigma}_{k}^{-1} =T_{b_{k_{2}}},     & \text{\em for  half twist $ \tilde{\sigma}_{k} $ about $ b_{k_{1}}, b_{k_{2}} $}   \\
	 \tilde{\sigma}_{k}  T_{b_{k_{2}}}=T_{b_{k_{2}}} \tilde{\sigma}_{k}. & \text{\em otherwise}\\
	\end{array}
\end{equation}
\end{thm}

\begin{proof}
According to Lemma \ref{basicdt} and Lemma \ref{lanternlem},
one can check these lifts also satisfy the relations in Lemma \ref{Impresentation}.
Besides, according to Lemma \ref{kerlem}, $ \text{Ker}\pi $ is the free abelian group generated by $ \{T_{b_{k}}, b_{k} \in \mathcal{B} \} $, where $ T_{b_{k}} $ is the $ 1/m $ twist about $ b_{k} $.

We have
$ \tilde{a}_{ij}T_{b_{k}}=T_{b_{k}}\tilde{a}_{ij} $, for $ a_{i,j} \in \text{Im}\pi $ and $ k \in [r] $.
And if $ \tilde{\sigma}_{k} $ is a half twist about boundary componets $ b_{k_{1}}, b_{k_{2}} $, then $ \tilde{\sigma}_{k}  T_{b_{k_{1}}} \tilde{\sigma}_{k}^{-1} =T_{b_{k_{2}}}   $, $ \tilde{\sigma}_{k} T_{b_{k_{2}}} \tilde{\sigma}_{k}^{-1} =T_{b_{k_{1}}}    $. Otherwise $ \tilde{\sigma}_{k}  T_{b_{k_{2}}}=T_{b_{k_{2}}} \tilde{\sigma}_{k} $.

Thus applying Lemma \ref{basicgp} and (\ref{quosq}), we obtain the presentation of $\mathcal{MCG}(S,M) $ as in Theorem \ref{mcgSg0prop}.

\end{proof}

\subsection{Presentations of mapping class groups stabilizing boundaries for genus $ g\ge1 $} \quad

Next, we calculate $ \mathcal{MCG}(S,M)$ for genus $ g\ge1 $ based on (\ref{quosq}).

\begin{prop}\cite{LP}\label{Sr1toSr0prop}
	There is a homomorphism $ \varphi : \text{\em Mod}(S_{g,1}, \mathcal{P}_{n}) \to \text{\em Mod}(S_{g,0}, \mathcal{P}_{n})$ induced by the inclusion of $ S_{g,1} $ to $  S_{g,0}  $. Then $ \varphi  $ is epimorphic, and
	
	(i) \, if $ g\ge 2 $, the kernel is generated by $ \{x_{n}^{-1}x_{n}'\} $, where $ x_{n} $ and $ x_{n}' $ are Dehn twists about $ \alpha_{n} $ and $ \alpha_{n}' $ , respectively, see Figure \ref{Sr1toSr0kernalfig};
	
	(ii) \, if $ g= 1 $, the kernel is generated by $ \{x_{n}^{-1}x_{n}',e^{-1}e'\} $, where $ e $ and $ e' $ are Dehn twists about $ \delta $ and $ \delta' $, respectively, see Figure \ref{Sr1toSr0kernalfig}.
	
	Moreover, following the notations in Theorem \ref{gr1pprethm}, we have the following relations in $ S_{g,1} $:
	\[ \begin{array}{lll}
		x_{n}&=&x_{0}^{1-n}\Delta(x_{1},v_{1},..., v_{n-1})=x_{0}^{1-n}(x_{1}v_{1}\cdots v_{n-1})^{n}, \\
		x_{n}'&=& \left\{ \begin{array}{ll}
			x_{0}^{3-2g}\Delta(z,y_{2}, ..., y_{2g-1}), & g\ge 2 \\
			x_{0},  & g=1\\
		\end{array}\right.\\
	e&=& \Delta^{2}(v_{1}, ..., v_{n-1})=(v_{1}\cdots v_{n-1})^{n}, \,\,\, g=1\\
	e'&=& \Delta^{4}(x_{0},y_{1}),        \,\,\, g=1
	\end{array} \]
where $ v_{i},i \in [n-1] $ are half twists about $ \tau_{i} $ in Figure \ref{genb1fig}.
\end{prop}
	\begin{figure}[htbp]
	\centering
	\subfigure{
\begingroup%
  \makeatletter%
  \providecommand\color[2][]{%
    \errmessage{(Inkscape) Color is used for the text in Inkscape, but the package 'color.sty' is not loaded}%
    \renewcommand\color[2][]{}%
  }%
  \providecommand\transparent[1]{%
    \errmessage{(Inkscape) Transparency is used (non-zero) for the text in Inkscape, but the package 'transparent.sty' is not loaded}%
    \renewcommand\transparent[1]{}%
  }%
  \providecommand\rotatebox[2]{#2}%
  \newcommand*\fsize{\dimexpr\f@size pt\relax}%
  \newcommand*\lineheight[1]{\fontsize{\fsize}{#1\fsize}\selectfont}%
  \ifx\svgwidth\undefined%
    \setlength{\unitlength}{537.46570113bp}%
    \ifx\svgscale\undefined%
      \relax%
    \else%
      \setlength{\unitlength}{\unitlength * \real{\svgscale}}%
    \fi%
  \else%
    \setlength{\unitlength}{\svgwidth}%
  \fi%
  \global\let\svgwidth\undefined%
  \global\let\svgscale\undefined%
  \makeatother%
  \begin{picture}(0.5,0.16)%
    \lineheight{1}%
    \setlength{\unitlength}{8cm}
    \setlength\tabcolsep{0pt}%
    \put(0,0){\includegraphics[width=\unitlength,page=1]{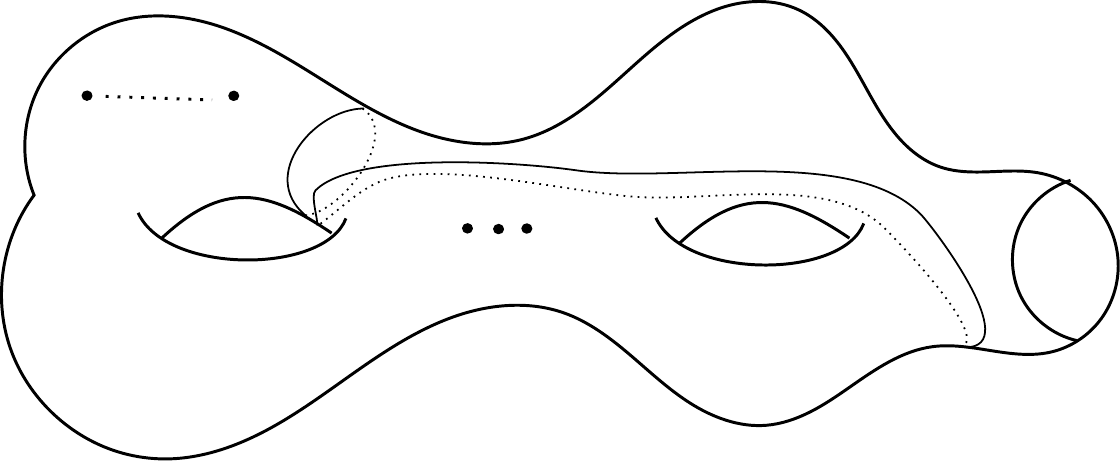}}%
    \put(0.32056788,0.34263909){\makebox(0,0)[lt]{\lineheight{1.25}\smash{\begin{tabular}[t]{l}$\alpha_n$\end{tabular}}}}%
    \put(0.85463788,0.05471856){\makebox(0,0)[lt]{\lineheight{1.25}\smash{\begin{tabular}[t]{l}$\alpha_n'$\end{tabular}}}}%
  \end{picture}%
\endgroup%

}%
\subfigure{
\begingroup%
  \makeatletter%
  \providecommand\color[2][]{%
    \errmessage{(Inkscape) Color is used for the text in Inkscape, but the package 'color.sty' is not loaded}%
    \renewcommand\color[2][]{}%
  }%
  \providecommand\transparent[1]{%
    \errmessage{(Inkscape) Transparency is used (non-zero) for the text in Inkscape, but the package 'transparent.sty' is not loaded}%
    \renewcommand\transparent[1]{}%
  }%
  \providecommand\rotatebox[2]{#2}%
  \newcommand*\fsize{\dimexpr\f@size pt\relax}%
  \newcommand*\lineheight[1]{\fontsize{\fsize}{#1\fsize}\selectfont}%
  \ifx\svgwidth\undefined%
    \setlength{\unitlength}{268.10476389bp}%
    \ifx\svgscale\undefined%
      \relax%
    \else%
      \setlength{\unitlength}{\unitlength * \real{\svgscale}}%
    \fi%
  \else%
    \setlength{\unitlength}{\svgwidth}%
  \fi%
  \global\let\svgwidth\undefined%
  \global\let\svgscale\undefined%
  \makeatother%
  \begin{picture}(0.4,0.32)%
    \lineheight{1}%
    \setlength{\unitlength}{4cm}
    \setlength\tabcolsep{0pt}%
    \put(0,0){\includegraphics[width=\unitlength,page=1]{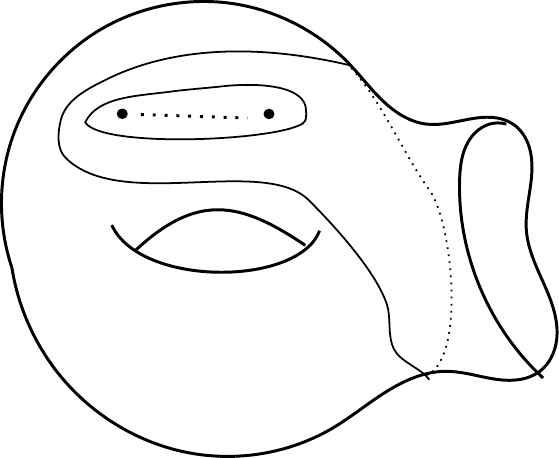}}%
    \put(0.57835006,0.58385898){\makebox(0,0)[lt]{\lineheight{1.25}\smash{\begin{tabular}[t]{l}$\delta$\end{tabular}}}}%
    \put(0.59949691,0.19052685){\makebox(0,0)[lt]{\lineheight{1.25}\smash{\begin{tabular}[t]{l}$\delta'$\end{tabular}}}}%
  \end{picture}%
\endgroup%

}%
\caption{Generators of kernel of $ \varphi $}
	\label{Sr1toSr0kernalfig}
\end{figure}
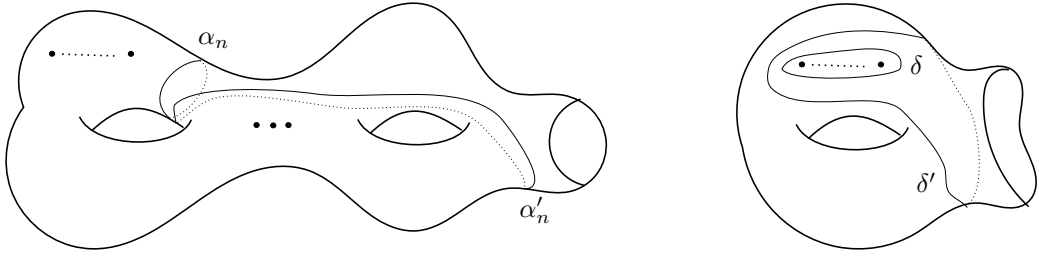

Restricting $ \varphi $ to $ \text{PMod}(S_{g,1},\mathcal{P}_{n}) $, since $ x_{n}, x_{n}', e, e' \in \text{PMod}(S_{g,1},\mathcal{P}_{n}) $, we obtain a presentation of $ \text{PMod}(S_{g,0},\mathcal{P}_{n}) $.

\begin{lem}\cite{LP} \label{dthtrelationlem}
	Let $ x_{i}, i=1, 2, 3 $ be Dehn twists about $ \alpha_{i} $ and $ v $ the half twist about $ \tau $, see Figure \ref{Sg0r2relationfig}. Then we have
	\begin{equation*}
		vx_{1}vx_{1}=x_{0}x_{2}.
	\end{equation*}
\end{lem}

\begin{figure}[htbp]
	\begin{tikzpicture}[scale=1]
		\draw (0,-0.5) arc (-90:90:0.35 and 0.5)
		(1,-0.5) arc (-90:90:0.35 and 0.5)
		(-1,-0.5) arc (-90:90:0.35 and 0.5)
		;
		\draw[line width=1] (-1.5,0) ellipse (0.35 and 0.5)
		(-1.5,0.5)--(1.5,0.5) (-1.5,-0.5)--(1.5,-0.5)
		(1.5,-0.5) arc (-90:90:0.35 and 0.5)
		;
		\draw[dashed, line width=1] (1.5,0.5) arc (90:270:0.35 and 0.5)
		;
		\draw[dashed] (1,0.5) arc (90:270:0.35 and 0.5)
		(-1,0.5) arc (90:270:0.35 and 0.5)
		(0,0.5) arc (90:270:0.35 and 0.5)
		;
		
		\filldraw  (-0.3,0) circle [radius=1pt]
		(0.8,0) circle [radius=1pt]
		;
		\draw (-0.3,0)--(0.8,0);
		
		\node(a0) at (-1,0.65) {$ \alpha_{0} $};
		\node(a1) at (0,0.65) {$ \alpha_{1} $};
		\node(a2) at (1,0.65) {$ \alpha_{2} $};
		\node(t) at (0.15,0.15) {$ \tau $};

	\end{tikzpicture}
	\caption{A relation in $ \text{Mod}(S_{0,2}, \mathcal{P}_{2}) $}	
	\label{Sg0r2relationfig}
\end{figure}
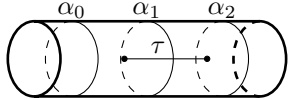

\begin{lem} \label{aijrelationlem}
	Let $a_{ij} $ be the Dehn twist about $ \alpha_{ij} $ in Figure \ref{PSgrelationfig}, $ y_{1} $ the Dehn twist about $ \beta_{1} $ in Figure \ref{abparitalfig}, $ x_{k} $   Dehn twists about $ \alpha_{k}, k=i-1, i, j-1, j $ in Figure \ref{abparitalfig}. Then we have
	
	\begin{equation*}
		a_{ij} =x_{i-1}s_{i+1j}x_{i}^{-1}s_{ij}^{-1},
	\end{equation*}
\end{lem}
where
	\[ s_{kj}=  y_{1}x_{k-1}x_{j}y_{1} x_{j-1} (y_{1}x_{k-1}x_{j}y_{1})^{-1}, \,\,\, k=i,i+1.
\]

\begin{figure}[htbp]
	\centering
		\begin{minipage}[t]{0.5\linewidth}
			\begin{tikzpicture}[scale=1.5]
				\draw (0,0) ellipse (2 and 1);
				
				\filldraw  (-1,0.5) circle [radius=1pt]
				(-0.3,0.5) circle [radius=1pt]
				(0.3,0.5) circle [radius=1pt]
				(1,0.5) circle [radius=1pt];
				
				\draw[dotted] (0.4,0.5) -- (0.9,0.5) (-0.9,0.5) -- (-0.4,0.5) (-0.2,0.5) -- (0.2,0.5) ;
				
				\draw (-0.3,0.7) ..controls (-0.1,0.8) and (-0.2,0.1)  .. (0,0.2)
				(-0.3,0.7) ..controls (-0.6,0.8) and (-0.3,-0.1)  .. (0,0)
				(0.3,0.7) ..controls (0.1,0.8) and (0.2,0.1)  .. (0,0.2)
				(0.3,0.7) ..controls (0.6,0.8) and (0.3,-0.1)  .. (0,0);

				\node(p1) at (-1,0.58) {$ p_{1} $};
				\node(pi) at (-0.3,0.58) {$ p_{i} $};
				\node(pj) at (0.3,0.58) {$ p_{j} $};
				\node(pn) at (1,0.58) {$ p_{n} $};
				\node(aij) at (0,-0.13) {$ \alpha_{ij} $};
				
				\draw (-1.2,-0.4) ..controls (-1.1,-0.2) and (-0.8,-0.2)  .. (-0.7,-0.4)
				(-1.25,-0.3) ..controls (-1.15,-0.55) and (-0.75,-0.55)  .. (-0.65,-0.3)
				
				(0.1,-0.4) ..controls (0.2,-0.2) and (0.5,-0.2)  .. (0.6,-0.4)
				(0.05,-0.3) ..controls (0.15,-0.55) and (0.55,-0.55)  .. (0.65,-0.3);
				
				\node(dot) at (-0.3,-0.4) {$ \cdots $};
			\end{tikzpicture}
			\caption{Loops $ \alpha_{ij} $ in $ (S_{g,0}, \mathcal{P}_{n}) $}	
			\label{PSgrelationfig}
		\end{minipage}%
		\begin{minipage}[t]{0.5\linewidth}
			\input{Figures/aij3.pdf_tex}
			\caption{Loops $ \alpha_{i} $ and $ \beta_{1} $ in $ (S_{g,0}, \mathcal{P}_{n}) $}
			\label{abparitalfig}
		\end{minipage}%
	\centering
\end{figure}

\begin{proof}
	Let $ s_{ij} $, $ s_{i+1j} $ be Dehn twists about $ \zeta_{ij} $ and $ \zeta_{i+1j} $ in Figure \ref{aijrep}. According to Lemma \ref{basicdt} one can check that in $ \text{PMod}(S_{g,0},\mathcal{P}_{n}) $,
	\[ s_{ij}=  y_{1}x_{i-1}x_{j}y_{1} x_{j-1} (y_{1}x_{i-1}x_{j}y_{1})^{-1},
\]
\[ s_{i+1j}=   y_{1}x_{i}x_{j}y_{1} x_{j-1} (y_{1}x_{i}x_{j}y_{1})^{-1}. \]
	Besides, according to Lemma \ref{lanternlem}, we have
	\[  a_{ij}s_{1}x_{i}=x_{i-1}s_{2}. \]
	Thus \[
		a_{ij} =x_{i-1}s_{i+1j}x_{i}^{-1}s_{ij}^{-1}. \]
		
\end{proof}

\begin{figure}[htbp]
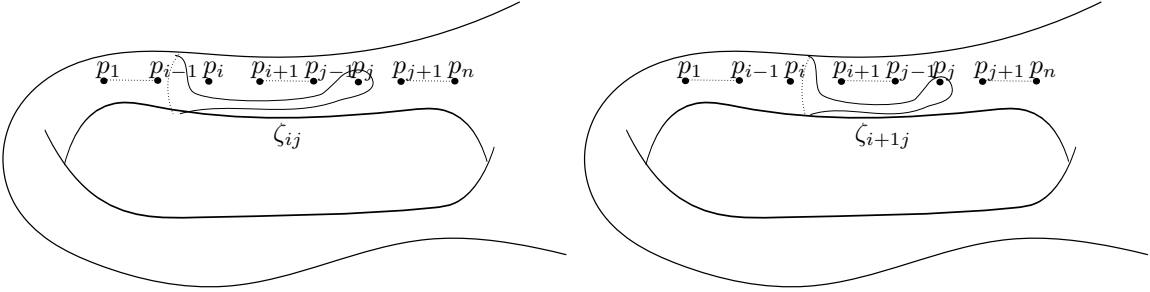

	\centering
	\subfigure{
		\input{Figures/aij1.pdf_tex}
	}%
	\subfigure{
		\input{Figures/aij2.pdf_tex}
	}%
	\caption{Loops $ \zeta_{ij} $ and $ \zeta_{i+1j} $ in $ (S_{g,0}, \mathcal{P}_{n}) $}
	\label{aijrep}
\end{figure}

By Lemma \ref{basicdt}, we have
\begin{equation}\label{sijrelation}
	s_{ii}=x_{i},  \,\,\,  v_{j}s_{ij}v_{j}^{-1}=s_{ij+1}, \text{ and } v_{k}s_{ij}=s_{ij}v_{k}, k\neq i-1, j-1, j,
\end{equation}
where $ v_{i},i \in [n-1] $ are half twists about $ \tau_{i} $ in Figure \ref{genb1fig}.

According to Lemma \ref{aijrelationlem} and (\ref{sijrelation}), we have
\begin{equation}\label{vi2rep}
	 v_{i}^{2}= a_{ii+1}=x_{i-1}s_{i+1i+1}x_{i}^{-1}s_{ii+1}^{-1}=x_{i-1}x_{i+1}x_{i}^{-1}s_{ii+1}^{-1},
\end{equation}

\begin{equation}\label{prodvrep}
	\begin{array}{lll}
		\Delta^{2}(v_{1}, ..., v_{n-1})&=&(v_{1}\cdots v_{n-1})^{n} \\
		&=&(a_{12}\cdots a_{1n})\cdots(a_{n-2n-1}a_{n-2n})(a_{n-1n} )\\
	\end{array}
\end{equation}

\begin{equation}\label{prodxvrep}
	\begin{array}{lll}
		\Delta(x_{1},v_{1},..., v_{n-1}) &=&x_{0}^{1-n}(x_{1}v_{1}\cdots v_{n-1})^{n}\\
		&=&x_{0}^{1-n}(s_{11}v_{1}\cdots v_{n-1})^{n} \\
		&=&x_{0}^{1-n}s_{11}s_{12}(v_{1}\cdots v_{n-1})^{2}(s_{11}v_{1}\cdots v_{n-1})^{n-2}   \\
		&=&x_{0}^{1-n}s_{11}s_{12}s_{13}(v_{1}\cdots v_{n-1})^{3}(s_{11}v_{1}\cdots v_{n-1})^{n-3}  \\
		&=&  \cdots\cdots\\
		&=& x_{0}^{1-n} s_{11}\cdots s_{1n} (v_{1}\cdots v_{n-1})^{n}\\
		&=& x_{0}^{1-n} s_{11}\cdots s_{1n}(a_{12}\cdots a_{1n})\cdots(a_{n-2n-1}a_{n-2n})(a_{n-1n} )\\
		
	\end{array}
\end{equation}

Suppose  $S=S_{g,r} $ has $ n_{k} $ boundary components with $ k $ marked points, the set $  \mathcal{B} $ of boundary components, the set
$ \mathcal{P}_{m} $ of punctures and the set $ M $ of marked points ,
the set $  \mathcal{B} $ is ordered by the number of marked points as given in (\ref{order}) with a map  $ \omega : \mathcal{P}_{m} \cup \mathcal{B}  \to  \mathcal{P}_{m+r}$  keeping the order for $m+r=n  $,
 and $\widetilde{S}  =S / \mathcal{B}=S_{g,0}  $  is the quotient space. Let $ I = [n-1]\backslash \{m, m+\Sigma_{i=1}^{j}n_{i} | j\in \mathbb{N}  \}  $.

Each generator $ \text{Per}_{i} $ of $ \Sigma_{S} $ in has a lift $ v_{i} $ in $  \text{Mod}(S_{g,0},\mathcal{P}_{n}) $, where $ v_{i} $ is the half twist about $ \tau_{i} $ in Figure \ref{genb1fig}. According to Lemma \ref{dthtrelationlem} and (\ref{vi2rep}), we obtain the relations of commutations and the lifts of relations in $ \Sigma_{S}  $.
Thus applying Lemma \ref{basicgp} to the sequence (\ref{imsq}), we obtain a presentation of $ \text{Im}\pi $ for genus $ g \ge 1 $.

Following the notations in Theorem \ref{gr1pprethm}, let $ A(\Gamma_{g,0,n}) $ be the Artin group associated with $ \Gamma_{g,0,n} $ (see Figure \ref{Sr0coxgrafig}), $ A' $ the subgroup generated by $ x_{i}, 0 \le i \le n $, $ y_{j}, j \in [2g-1] $, $ z $ and $ v_{k} $, $ k \in I $.  Then $ \text{Im}\pi $ is isomorphic to the quotient of $ A' $  with the following relations.

\textbullet \, Relations from $ \text{PMod}(S_{g,1},\mathcal{P}_{n})  $ in Theorem \ref{gr1pprethm}:
\[ \Delta^{4}(y_{1},y_{2},y_{3},z)=\Delta^{2}(x_{0},y_{1},y_{2},y_{3},z), \,\,\,  g \ge 2 \]
\[ \Delta^{2}(y_{1},y_{2},y_{3},y_{4},y_{5},z) =\Delta(x_{0},y_{1},y_{2},y_{3},y_{4},y_{5},z).  \,\,\, g \ge 3 \]
\[ x_{k}\Delta^{-1}(x_{i+1,x_{j},y_{1}})x_{i}\Delta(x_{i+1,x_{j},y_{1}})=\Delta^{-1}(x_{i+1,x_{j},y_{1}})x_{i}\Delta(x_{i+1,x_{j},y_{1}}) x_{k}, \,\,\,0\le k <j<i\le n-1  \]
\[ y_{2}\Delta^{-1}(x_{i+1,x_{j},y_{1}})x_{i}\Delta(x_{i+1,x_{j},y_{1}})=\Delta^{-1}(x_{i+1,x_{j},y_{1}})x_{i}\Delta(x_{i+1,x_{j},y_{1}})y_{2}, \,\,\, 0 \le j < i \le n-1, g\ge 2  \]
\[ \Delta(x_{0},x_{1},y_{1},y_{2},y_{3},z)=\Delta^{2}(x_{1},y_{1},y_{2},y_{3},z),\,\,\, g\ge 2 \]
\[ \Delta(x_{i},x_{i+1},y_{1},y_{2},y_{3},z)\Delta^{-2}(x_{i+1},y_{1},y_{2},y_{3},z)=\Delta(x_{0},x_{i},x_{i+1},y_{1})\Delta^{-2}(x_{0},x_{i+1},y_{1}), 1 \le i \le n-1. g\ge 2 \]

\textbullet \, Relations from the kernal in Proposition \ref{Sr1toSr0prop}:
\[ \begin{array}{c}
	x_{0}^{n}=x_{0}^{1-n} s_{11}\cdots s_{1n}(a_{12}\cdots a_{1n})\cdots(a_{n-2n-1}a_{n-2n})(a_{n-1n} ), \text{  } g=1   \\
	\Delta^{4}(x_{0},y_{1})=(a_{12}\cdots a_{1n})\cdots(a_{n-2n-1}a_{n-2n})(a_{n-1n} ), \text{  } g=1 \\
x_{0}^{2-2g+n}	\Delta(z,y_{2}, ..., y_{2g-1})=x_{0}^{1-n} s_{11}\cdots s_{1n}(a_{12}\cdots a_{1n})\cdots(a_{n-2n-1}a_{n-2n})(a_{n-1n} ).  \text{  } g\ge 2 \\
\end{array} \]

\textbullet \, Other relations:
\[ v_{i}^{2}= a_{ii+1}. \]
Here
\[ a_{ij} =x_{i-1}s_{i+1j}x_{i}^{-1}s_{ij}^{-1}, \,\,\, s_{ij}=  y_{1}x_{i-1}x_{j}y_{1} x_{j-1} (y_{1}x_{i-1}x_{j}y_{1})^{-1}. \]
\begin{figure}[htbp]
	\[ 	\xymatrix{
		 v_{1}\ar@{-}[r]  &  v_{2} \ar@{-}[r] &  \ar@{.}[r] & \ar@{-}[r] & v_{n-1}   & \\
		  x_{1} \ar@{-}[dr] \ar@{-}[u]^{4} & x_{2}  \ar@{-}[d]  \ar@{-}[u]^{4} \ar@{.}[r] &  x_{n-1}  \ar@{-}[dl]  \ar@{-}[urr]^{4}   \\
		x_{n} \ar@{-}[r] &  y_{1}\ar@{-}[r] & y_{2} \ar@{-}[r] & y_{3} \ar@{-}[r]& y_{4} \ar@{-}[r] & \ar@{.}[r] &  \ar@{-}[r] & y_{2g-1}\\	
		  & x_{0} \ar@{-}[u]&  &  z \ar@{-}[u]  \\
	} \]
	\caption{Coxeter graph $ \Gamma_{g,0,n} $ associated with $ \text{Mod}(S_{g,0}, \mathcal{P}_{n}) $  }
	\label{Sr0coxgrafig}
\end{figure}
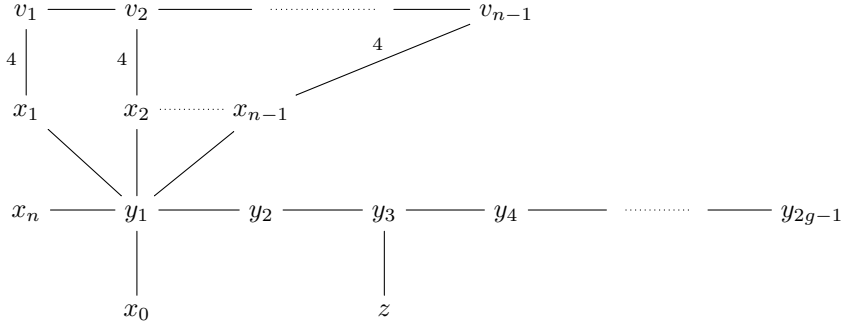

Applying Lemma \ref{basicgp} to the sequence (\ref{quosq}), we obtain a presentation of $ \mathcal{MCG}(S,M) $ for genus $ g \ge 1 $.

\begin{thm}\label{mcgSg1prop}
	 For genus $ g \ge 1 $,  suppose  $S=S_{g,r} $ has $ n_{i} $ boundary components with $ i $ marked points, the set $  \mathcal{B} $ of boundary components, the set of punctures
	$ \mathcal{P}_{m} $ and the set $ M $ of marked points, the set $  \mathcal{B} $ is ordered by the number of marked points as given in (\ref{order}) with a map keeping order: $ \omega : \mathcal{P}_{m} \cup \mathcal{B}  \to  \mathcal{P}_{m+r}$ for $m+r=n  $,  and
	$ I = [n-1]\backslash \{m, m+\Sigma_{i=1}^{j}n_{i} | j\in \mathbb{N}  \}  $.

	Let $ A(\Gamma_{g,0,n}) $ be the Artin group associated with $ \Gamma_{g,0,n} $ (see Figure \ref{Sr0coxgrafig}), $ A' $ the subgroup generated by $ x_{i}, 0 \le i \le n $, $ y_{j}, j \in [2g-1] $, $ z $ and $ v_{k} $, $ k \in I $. Suppose $ T_{b_{l}} ,b_{l} \in \mathcal{B} $ are $ 1/m $ twists about $ b_{l} $. Then $ \mathcal{MCG}(S,M) $ is isomorphic to the quotient of $ A'\langle T_{b_{l}}, b_{l} \in \mathcal{B} \rangle $ with the following relations:
	
	\textbullet \, Relations from $ \text{\em Im}\pi  $ :
	\begin{equation}\label{g1relation1}
		\begin{array}{c}
		\Delta^{4}(y_{1},y_{2},y_{3},z)=\Delta^{2}(x_{0},y_{1},y_{2},y_{3},z), \,\,\,  g \ge 2 \\
		\Delta^{2}(y_{1},y_{2},y_{3},y_{4},y_{5},z) =\Delta(x_{0},y_{1},y_{2},y_{3},y_{4},y_{5},z).  \,\,\, g \ge 3 \\
		 x_{k}\Delta^{-1}(x_{i+1,x_{j},y_{1}})x_{i}\Delta(x_{i+1,x_{j},y_{1}})=\Delta^{-1}(x_{i+1,x_{j},y_{1}})x_{i}\Delta(x_{i+1,x_{j},y_{1}}) x_{k}, \,\,\,0\le k <j<i\le n-1  \\
		 y_{2}\Delta^{-1}(x_{i+1,x_{j},y_{1}})x_{i}\Delta(x_{i+1,x_{j},y_{1}})=\Delta^{-1}(x_{i+1,x_{j},y_{1}})x_{i}\Delta(x_{i+1,x_{j},y_{1}})y_{2}, \,\,\, 0 \le j < i \le n-1, g\ge 2  \\
		 \Delta(x_{0},x_{1},y_{1},y_{2},y_{3},z)=T_{1}\Delta^{2}(x_{1},y_{1},y_{2},y_{3},z),\,\,\, g\ge 2 \\
		 \Delta(x_{i},x_{i+1},y_{1},y_{2},y_{3},z)\Delta^{-2}(x_{i+1},y_{1},y_{2},y_{3},z)=T_{i+1}\Delta(x_{0},x_{i},x_{i+1},y_{1})\Delta^{-2}(x_{0},x_{i+1},y_{1}), 1 \le i \le n-1, g\ge 2 \\
			x_{0}^{n}=x_{0}^{1-n} s_{11}\cdots s_{1n}(a_{12}\cdots a_{1n})\cdots(a_{n-2n-1}a_{n-2n})(a_{n-1n} ), \text{  } g=1   \\
			\Delta^{4}(x_{0},y_{1})=(a_{12}\cdots a_{1n})\cdots(a_{n-2n-1}a_{n-2n})(a_{n-1n} ), \text{  } g=1 \\
			x_{0}^{2-2g+n}	\Delta(z,y_{2}, ..., y_{2g-1})=x_{0}^{1-n} s_{11}\cdots s_{1n}(a_{12}\cdots a_{1n})\cdots(a_{n-2n-1}a_{n-2n})(a_{n-1n} ),  \text{  } g\ge 2 \\
			v_{i}^{2}= a_{ii+1}. \\
		\end{array}
	\end{equation}
		\textbullet \, Relations of commutations:
		\begin{equation}\label{g1relation2}
			\begin{array}{ll}
				T_{b_{i}}T_{b_{j}}=T_{b_{j}}T_{b_{i}}, & b_{i},b_{j} \in \mathcal{B}\\
				T_{b_{k}}X=XT_{b_{k}},  &b_{k} \in \mathcal{B}, X=x_{i},y_{j},z, 0\le i\le n, j\in [2g-1] \\
				v_{k}  T_{b_{k_{1}}}v_{k}^{-1} =T_{b_{k_{2}}},   & \text{\em for  half twist $ v_{k} $ about $ b_{k_{1}}, b_{k_{2}} $}\\
				v_{k}  T_{b_{k_{2}}}=T_{b_{k_{2}}} v_k, &\text{\em Otherwise}
			\end{array}
		\end{equation}
	where $ T_{i} $ is the Dehn twist about $ i $-th puncture or boundary component. In the case that it is a puncture, $ T_{i}=1 $, and in the case that it is a boundary component $ b_{k} $ with $ m $ marked points, $ T_{i}=T_{b_{k}}^{m} $, and
	\[ a_{ij} =x_{i-1}s_{i+1j}x_{i}^{-1}s_{ij}^{-1}, \,\,\, s_{ij}=  y_{1}x_{i-1}x_{j}y_{1} x_{j-1} (y_{1}x_{i-1}x_{j}y_{1})^{-1}. \]
\end{thm}
\begin{proof}
	For each generator in $ \text{Im}\pi $, there is a lift in $\mathcal{MCG}(S,M)$: In the case that $ x_{i}, y_{j},z $ are Dehn twist about $ \alpha_{i}, \beta_{j}, \gamma $ in Figure \ref{genb1fig}, respectively, the lifts are themselves. If $ v_{i} $ is a half twist and $ i $-th and $ i+1 $-th punctures in quotient space corresponding to  boundary component $ b_{k_{1}}, b_{k_{2}} $, then the lift is the half twist about $ b_{k_{1}}, b_{k_{2}} $, which is also denoted by $ v_{i} $; otherwise  the lift is itself.
	
	According to Lemma \ref{funelmtoboundlem}, one can check the relations of lifts are exactly the relations given above.
	
	Besides, according to Lemma \ref{kerlem}, $ \text{Ker}\pi $ is the free abelian group generated by $ \{T_{b_{k}}, b_{k} \in \mathcal{B} \} $, where $ T_{b_{k}} $ is the $ 1/m $ twist about $ b_{k} $.
	By Lemma \ref{basicdt}, $ T_{b_{k}} $ commutes with $ x_{i}, y_{j}, z $ and $ v_{i} $ which is not the half twist about $ b_{k} $.
	And if $ v_{k} $ is a half twist about boundary componets $ b_{k_{1}}, b_{k_{2}} $, then $ v_{k}  T_{b_{k_{1}}} v_{k}^{-1} =T_{b_{k_{2}}}   $, $ v_{k} T_{b_{k_{2}}} v_{k}^{-1} =T_{b_{k_{1}}}    $.  Thus according to Lemma \ref{basicgp}, we obtain the presentation of
	$ \mathcal{MCG}(S,M) $.

\end{proof}

\section{Application to cluster automorphism groups of cluster algebras}
 In this section, we apply the results on mapping class groups to cluster automorphism groups of cluster algebras based on the works in \cite{Gu,TBIS,BY}.

\subsection{Preliminaries} \quad

Let $ \mathcal{F}  $ be the field of rational polynomials over $ \mathbb{Q} $ in $ m $ independent variables, the \textbf{cluster algebra} $ \mathcal{A} $ is a $ R- $subalgebra of $ \mathcal{F} $, where $ R=\mathbb{Q}[x_{n+1},...,x_{m}] $, the ring of  polynomials of the frozen variables.

\begin{defi}
	A seed in $ \mathcal{F} $ is a pair $ (\textbf{\em x}_{ex},\textbf{\em x}_{fr},\widetilde{B}) $ where
	
	(1) $ \textbf{\em x}_{ex}=(x_{1},...,x_{n}) , \textbf{\em x}_{fr}=(x_{n+1},...,x_{m})$ and $ x_{1},..., x_{m} $ are algebraically independent in $ \mathcal{F} $;
	
	(2) $ \widetilde{B} $ is an $ m \times n $ extended skew-symmetrizable integer matrix, i.e., there exists a diagonal non-negative integer matrix $ D $ such that $ (DB)^{\top}=-DB $, where $ B $ is the upper $ n \times n $ submatrix of $ \widetilde{B}$, called the \textbf{exchange matrix};  $\widetilde{B} $ is  called the \textbf{extended exchange matrix} and  $ D $ is called the\textbf{ skew-symmetrizer}. Meanwhile, the lower $ n \times n $ submatrix of $ \widetilde{B}$, denoted as $C$, is called the {\bf $C$-matrix}.
	
	(3) $\textbf{\em x}= \textbf{\em x}_{ex} \cup \textbf{\em x}_{fr} $ is called the \textbf{cluster} of the  seed $ (\textbf{\em x}_{ex},\textbf{\em x}_{fr},\widetilde{B}) $. The elements of $ \textbf{\em x}_{ex}  $ are called \textbf{exchange variables} and the elements of $ \textbf{\em x}_{fr} $ are called \textbf{frozen variables}.  The exchange variables and  frozen variables are collectively called \textbf{cluster variables}. The monomials consisting of cluster variables (or exchange variables) in a cluster are called \textbf{extended cluster monomials} (or \textbf{cluster monomials}).
\end{defi}

Note that if $ B $ is skew-symmetric, then $ \widetilde{B}=(b_{ij}) $ can be represented by a quiver: its vertices labelled by cluster variables and $ x_{i}, x_{j} $ are joined by $ b_{ij}  $ edges if $ b_{ij}>0 $.

\begin{defi}
	Let $  (\textbf{\em x}_{ex},\textbf{\em x}_{fr},\widetilde{B}) $ be a seed in $  \mathcal{F}  $, $ k \in \{1,...,n \} $ The\textbf{ seed mutation} $ \mu_{k} $ transforms $(\textbf{\em x}_{ex},\textbf{\em x}_{fr},\widetilde{B})  $ to the new seed $ \mu_{k}(\textbf{\em x}_{ex},\textbf{\em x}_{fr},\widetilde{B}) = (\textbf{\em x}_{ex}',\textbf{\em x}_{fr}',\widetilde{B'})$ defined as follows:
	
	(1) $ \widetilde{B'}=(b_{ij}')_{m \times n} $, where
	$$ b_{ij}'= \left\{
	\begin{array}{ll}
		-b_{ij} &  \text{if } i=k \text{ or } j=k; \\
		b_{ij}+ sgn(b_{ik})max(b_{ik}b_{kj},0) & \text{otherwise}.
	\end{array}
	\right. $$
	
	(2) $ (\textbf{\em x}_{ex}',\textbf{\em x}_{fr}')=(x_{1}',...,x_{m}') $, where
	$$ x_{i}'= \left\{
	\begin{array}{ll}
		x_{i} &  \text{if } i \neq k ; \\
		x_{i}^{-1}(\prod_{b_{jk} > 0}x^{b_{jk}}_{j} + \prod_{b_{jk} <0 }x^{-b_{jk}}_{j} ) & \text{otherwise}.
	\end{array}
	\right. $$	
\end{defi}

In the sequel we also denote by  $ \mu_{x_{k}}= \mu_{k} $ and $ b_{x_{i}x_{j}}=b_{ij} $, for $ x_{k},x_{i} \in \textbf{x}, \, x_{j} \in \textbf{x}_{ex} $.

Let $ \mathcal{X} $ be the set of all cluster variables obtained from the initial seed $ \Sigma $  by finite steps of mutations, the {\bf cluster algebra} $ \mathcal{A}  $ is defined as the $ \mathbb{Q}- $subalgebra of $ \mathcal{F} $ generated by $ \mathcal{X} $. The {\bf rank} of $\mathcal A$ is defined as the number of exchange variables in any seed $\Sigma$.

We say a cluster algebra to be {\bf with coefficients}, if the set of frozen variables in the cluster algebra is  not  empty, otherwise we say the cluster algebra to be {\bf without coefficient}.

Cluster algebras discussed in this article are always supposed to be without coefficient.

\begin{defi}\cite{HuLi}
	Let $ \Sigma_{1}, \Sigma_{2} $ are  seeds of cluster algebras $ \mathcal{A}_{1} $ and $ \mathcal{A}_{2} $ respectively, where $ \Sigma_{1}=(\textbf{\em x}_{ex,1},\textbf{\em x}_{fr,1},$  $\widetilde{B}_{1,m_{1} \times n_{1}}), $  $ \Sigma_{2}=(\textbf{\em x}_{ex,2},\textbf{\em x}_{fr,2},\widetilde{B}_{2,m_{2} \times n_{2}}) $, $ \textbf{\em x}_{ex,i}=\{ x_{i,1},x_{i,2},...,x_{i,n_{i}} \},i=1,2 $ are exchange variables and $ \textbf{\em x}_{fr,i}= \{ x_{i,n+1},x_{i,n+2},...,$  $x_{i,m_{i}} \}, i=1,2 $ are frozen variables. Suppose $ f: \mathcal{A}_{1}\rightarrow \mathcal{A}_{2} $ is a morphism of algebras, if it satisfies the following conditions:
	
	(1) $ f(\textbf{\em x}_{ex,1}\sqcup \textbf{\em x}_{fr,1}) \subseteq \textbf{\em x}_{ex,2}\sqcup \textbf{\em x}_{fr,2}\sqcup \mathbb{Z} $;
	
	(2) $ f(\textbf{\em x}_{ex,1}) \subseteq \textbf{\em x}_{ex,2} \sqcup \mathbb{Z}  $;
	
	(3) If $ (y_{1},y_{2},...,y_{s}) $ is a $(f, \Sigma_{1}, \Sigma_{2})  $ bi-admissible sequence, i.e., $ y_{i}$ is an exchange variable in $ \mu_{y_{i-1}}...\mu_{y_{1}}(\Sigma_{1}) $ and $ f(y_{i}) $ is an exchange variable in $ \mu_{f(y_{i-1})}...$  $ \mu_{f(y_{1})}(f(\Sigma_{1})) $, $ i=1,...,s $, assume $ \mu_{y_{0}}(\Sigma_{1})=\Sigma_{1} $ and $ \mu_{f(y_{0})}(f(\Sigma_{1}))=f(\Sigma_{1}) $, we have
	\[ f(\mu_{y_{s}}...\mu_{y_{1}}(y))=\mu_{f(y_{s})}...\mu_{f(y_{1})}(f(y)), \]
	then we call $ f $ a \textbf{cluster morphism} bewteen $   \mathcal{A}_{1} $ and $  \mathcal{A}_{2} $.
\end{defi}

If a cluster morphism $ f $ bewteen $   \mathcal{A}_{1} $ and $  \mathcal{A}_{2} $ is monomorphic (epimorphic, isomorphic) as a morphism of algebras, we say it to be cluster monomorphic (epimorphic, isomorphic). In particular, the cluster isomorphism from $ \mathcal{A}_{1} $ to itself is said to be \textbf{cluster automorphism} and the group consisting of cluster automorphisms on $ \mathcal{A}_{1} $ is said to be \textbf{cluster automorphism group} of $ \mathcal{A}_{1} $, denoted by $ \text{Aut}\mathcal{A}_{1} $.

It following \cite{FST}, to define cluster algebras from surface, we require all surfaces to be not one of the following:

{\em A sphere with 1, 2 or 3 marked points, an unpunctured disk with 1,2 or 3 marked points on the boundary, or a once-punctured disk with 1 marked point on the boundary.}

Cluster algebras from surfaces are a particular type of cluster algebras whose exchange matrices are determined by a triangulation of surfaces.

\begin{thm}\cite{FST}
	(1) If $ (S,M) $ is not a closed surface with one puncture, then there are bijections:
	\[ \{ \text{tagged arcs in }  (S,M)  \} \to  \{\text{exchange variables in } \mathcal{A} \}, \]
	\[ \{ \text{tagged triangulations of }  (S,M) \} \to \{ c\text{lusters in } \mathcal{A} \}. \]
	
	(2) If $ (S,M) $ is a closed surface with one puncture, then there are bijections:
	\[ \{  \text{arcs in }  (S,M)  \} \to  \{\text{exchange variables in }  \mathcal{A} \}, \]
	\[ \{ \text{triangulations of }  (S,M) \} \to \{ \text{clusters in } \mathcal{A} \}. \]
	
	And in these bijections, a flip at $ \gamma $ coincide with the mutation at the corresponding exchange variable $ x_{\gamma} $.
\end{thm}

It should be noted that for a loop $ \zeta $ which cuts out a puncture $ p $, its corresponding element $ x_{\zeta} $ in the corresponding cluster algebra $ \mathcal{A} $ satisfies $ x_{\zeta}=x_{\gamma}x_{\gamma^{p}} $, where $ \gamma $ is the arc lying in the monogon bounded by $ \zeta $ and $ \gamma^{p} $ is the tagged arc obtained from $ \gamma $ via changing the tagging at $ p $.

\subsection{Presentations of cluster automorphism groups of cluster algebras from surfaces} \quad

Let $ (S,M) $ be a bordered surface with marked points, $ h $ be a homeomorphism of $ (S,M) $ such that $ h(M)=M $ and $ R  $ is a subset of the set consisting of all punctures in $ (S,M) $, then we can obtain a cluster automorphisms of $ \mathcal{A}(S,M) $ corresponding to $ (S,M) $ induced by $ h $ and $ R $:
\[ \psi_{h} \in \text{Aut}\mathcal{A}(S,M) \quad \text{defined by } \psi_{h}(x_{\gamma}):=x_{h(\gamma)} , \]
where $ \gamma \in A_{\bowtie}(S,M) $ the set consisting of all tagged arcs of  $ (S,M) $ and $ x_{\gamma} $ is its counterpart cluster variable in $ \mathcal{A}(S,M) $. Besides, we set
\[ \psi_{R} \in \text{Aut}\mathcal{A}(S,M) \quad \text{defined by } \psi_{R}(x_{\gamma}):=x_{\gamma^{R}}, \]
where $ \gamma^{R} $ denote the tagged arc obtained from $ \gamma $ by changing the taggings of it at those of its endpoints that belong to $ R $. For $ h $ and $ R $ as above, denote by
$ \psi_{h,R}:=\psi_{h}\psi_{R}. $

If $ h $ is homotopic to $ h' $, then for  $ \gamma \in A_{\bowtie}(S,M) $, $ h(\gamma)=h'(\gamma) $. Thus the above map induces a map from $ \mathcal{MCG}^{\pm}_{\bowtie}(S,M) $ to $ \text{Aut}\mathcal{A}(S,M)  $, denoting by $ \rho $. In addition, $ \rho $ is a monomorphism of groups.
According to \cite{Gu,TBIS,BY}, we have the following theorem:

\begin{thm}\label{mcg}\cite{Gu,TBIS,BY}
	Let $ (S,M) $ be a bordered surface with marked points different from any one of 		
	(1) the 4-punctured sphere, 		
	(2) the once-punctured 4-gon, 		
	(3) the twice-punctured digon. Then the group of cluster automorphisms of the corresponding cluster algebra $\mathcal A(S,M)$  is given as follows:
	
	(a)  if $ (S,M) $ is a once-punctured closed surface, then
	\[\mathcal{MCG}^{\pm}(S,M) \cong \text{\em Aut}\mathcal{A}(S,M), \; \mathcal{MCG}(S,M) \cong \text{\em Aut}^{+}\mathcal{A}(S,M)   \;\;\; via\;\; \;
	h \mapsto \psi_{h},
	\]
	
	(b)  if $ (S,M) $ is not a once-punctured closed surface, then
	\[\mathcal{MCG}^{\pm}_{\bowtie}(S,M) \cong \text{\em Aut}\mathcal{A}(S,M), \; \mathcal{MCG}_{\bowtie}(S,M) \cong \text{\em Aut}^{+}\mathcal{A}(S,M)   \;\;\; via\;\;\;
	(h,R) \mapsto \psi_{h,R},
	\]
	where $ h $ is a representative element of a class in $ \mathcal{MCG}^{\pm}(S,M)$ or $ \mathcal{MCG}(S,M) $, $ R  $ is a subset of sets of all punctures in $ (S,M)$. And $ \text{\em Aut}^{+}\mathcal{A}(S,M) $ is the subgroup of cluster automorphism groups consisting of direct cluster automorphisms which map the extended matrix of a seed to itself.
\end{thm}

 This theorem builds up a connection between mapping class groups and  cluster automorphism groups for surfaces  different from those in cases (1)--(3).
Both the results in the above existing literatures  and that obtained below show that the kind of these surfaces is very important and useful.
For convenience, we give the definition for such surfaces as follows:
\begin{defi}
	A bordered surface $ (S,M) $  with marked points is called a \textbf{feasible surface} if $ (S,M) $ is different from any one of
	(1) the 4-punctured sphere, (2) the once-punctured 4-gon, and 		
	(3) the twice-punctured digon.
\end{defi}	

So except for the special cases, we can consider the cluster automorphism group of cluster algebras from feasible surfaces as $ \mathcal{MCG}^{\pm}_{\bowtie}(S,M) $ or $ \mathcal{MCG}^{\pm}(S,M) $.

Since each bordered oriented surfaces $ (S,M) $ with marked points has an orientation-reversing homeomorphism $ \iota \in \text{Homeo}(S,M) $ such that $ \iota^{2}=id $ and $ \iota^{-1}T\iota=T^{-1}  $ for any Dehn twist or half twist. Thus the semiproduct in $ \mathcal{MCG}^{\pm}(S,M)=\mathcal{MCG}(S,M)\rtimes Z_{2} $ is given by:
\begin{equation}\label{semiprodz2}
	(h_{1},\varepsilon_{1})* (h_{2},\varepsilon_{2}) =( h_1 h_{2}^{(-1)^{\varepsilon_{1}}},\varepsilon_{1}+\varepsilon_{2}),
	\;\;\; (h_{1},\varepsilon_{1}), (h_{2},\varepsilon_{2}) \in \mathcal{MCG}(S,M)\rtimes Z_{2}.
\end{equation}

Applying the presentations of mapping class groups to Theorem \ref{mcg}, we obtain the presentations of cluster automorphism groups of cluster algebras from feasible surfaces.

\begin{cor}\label{preclusterautocor}  For a feasible surface $(S,M)$,
		suppose $S=S_{g,r} $ has the set $ \mathcal{P}_{n} $ of punctures and the set $ M $ of marked points, and  $ (S,\mathcal{P}_{n}) \neq( S_{0,2}, \emptyset)  $.
	
	 Then $\text{\em Aut}\mathcal A(S,M)$  is given as follows:
	
	(a)  if $ (S,M) $ is a once-punctured closed surface, then
	$ \text{\em Aut}\mathcal{A}(S,M) \cong  \mathcal{MCG}(S,M)\rtimes Z_{2},$
	
	(b)  if $ (S,M) $ is not a once-punctured closed surface, then
	$ \text{\em Aut}\mathcal{A}(S,M) \cong (\mathcal{MCG}(S,M)\rtimes Z_{2}) \ltimes \mathbb{Z}_{2}^{\mathcal{P}_{n}},$
	with the semiproduct given by
	\[
	(h_{1},\varepsilon_{1}, R_{1}) * (h_{2},\varepsilon_{2},R_{2}) =( h_1 h_{2}^{(-1)^{\varepsilon_{1}}},\varepsilon_{1}+\varepsilon_{2},h_{2}(R_{1}) \ominus R_{2} ),
	 \]
		 where the presentation of $ \mathcal{MCG}(S,M) $ is given by Theorem \ref{mcgSg0prop} for case genus $ g=0 $ and by Theorem \ref{mcgSg1prop} for case genus $ g \ge 0 $.   \end{cor}

 Note that in Corollary \ref{preclusterautocor}, the first semiproduct in $ (\mathcal{MCG}(S,M)\rtimes Z_{2}) \ltimes \mathbb{Z}_{2}^{\mathcal{P}_{n}} $ is given in (\ref{semiprodz2}) and the second semiproduct is given in (\ref{semiprodpunctures}).

If $S =S_{0,2} $ and $ \mathcal{P}_{n} =\emptyset $, suppose its two boundary component $ b_{1} $, $ b_{2} $ contain $ p $, $ q $ marked points respectively. According to \cite{ASS}, \begin{equation}\label{S02mcg}
	\mathcal{MCG}(S,M) = \left\{ \begin{array}{ll}
		H_{p,q} &  p \neq q;\\
		H_{p,q} \rtimes \mathbb{Z}_{2} & q=p,   \\
	\end{array} \right.
\end{equation}
where $ H_{p,q}= \langle r_{1}, r_{2} | \, r_{1}r_{2}=r_{2}r_{1},  r_{1}^{p}=r_{2}^{q}  \rangle $.

\subsection{Presentations of cluster automorphism groups for exceptional cases} \quad

It remains to calculate   the cluster automorphism group of a cluster algebra from a surface $ (S,M) $ which is not feasible, that is, which is one of (1) the 4-punctured sphere, 		
(2) the once-punctured 4-gon, 		
(3) the twice-punctured digon.

The case (2) or (3) means the cluster algebra  $\mathcal A(S,M)$ is of Dynkin type $ D_{4} $ or of Euclidean type $ \tilde{D}_{4} $. Thus according to \cite{ASS}, for the case (2), the cluster automorphism group \begin{equation}\label{D4cong}
	\text{Aut}\mathcal A(S,M) \cong Di_{4} \times \Sigma_{3},
\end{equation}
where  $ Di_4 $  stands for the dihedral group, and  $ \Sigma_3 $  stands for the permutation group;
for the case (3), it holds
\begin{equation}\label{tildeD4cong}
	\text{Aut}\mathcal A(S,M) \cong \mathbb{Z} \times S_{4} \rtimes \mathbb{Z}_{2}.
\end{equation}

As for the case (1), we have the following proposition.
\begin{thm}\label{Precase1}
	Suppose $ \mathcal{A}(S,M) $ is the cluster algebra from a 4-punctured sphere, then
	\begin{equation}\label{sph4cong}
		\text{\em Aut}\mathcal A(S,M) \cong  (\mathcal{MCG}(S,M)\rtimes Z_{2}) \ltimes \mathbb{Z}_{2}^{\mathcal{P}_{4}} \times \mathbb{Z}_{2}^{2}.
	\end{equation}
\end{thm}
\begin{proof}
	As mentioned above, $ \mathcal{MCG}_{\bowtie}^{\pm}(S,M)=(\mathcal{MCG}(S,M)\rtimes Z_{2}) \ltimes \mathbb{Z}_{2}^{\mathcal{P}_{4}} $.
	Firstly, we fix a triangulation $ \mathcal{T} $ without self-folded triangle of $ (S,M) $ as in  Figure \ref{sph4trifig} (a), and the associated quiver $ Q(\mathcal{T}) $ also showed in  Figure \ref{sph4trifig} (a) has maximal arrows in its mutation equivalence class, which is called the maximal triangulation of $ (S,M) $ in \cite{BY}.
	
	Consider the case that $ \phi \in \text{Aut}\mathcal A(S,M)  $ maps $ \mathcal{T} $ to a maximal triangulation and maps edges of triangles to edges of triangles. According to \cite{BY}, there exist $ f \in \mathcal{MCG}_{\bowtie}^{\pm}(S,M)  $ such that $ f $ restricted to arcs is equal to $ \phi $. Since $\mathcal{MCG}_{\bowtie}^{\pm}(S,M)   $ can be embedded into $	\text{Aut}\mathcal A(S,M)   $, the subgroup of $ \text{Aut}\mathcal A(S,M)  $ consisting of such elements is isomorphic to $ \mathcal{MCG}_{\bowtie}^{\pm}(S,M)  $.
	
	Suppose $ \phi \in 	\text{Aut}\mathcal A(S,M)  $ maps $ \mathcal{T} $ to a maximal triangulation, but does not map edges of triangles to edges of triangles. After relabelling, $ \phi  $ must be the form as follows: $ \phi(\alpha_{i})=\beta_{i} $, $ i=1, ..., 6 $, see Figure \ref{sph4trifig}. Let $ \sigma_{\phi} $ be the permutation interchange $ \phi(\alpha_{2}) $ and $ \phi(\alpha_{4}) $, then $ \sigma \circ \phi $ is also a isomorphism from $  Q(\mathcal{T}) $ to $ Q(\mathcal{T}')  $ for some triangulation $ \mathcal{T}' $. Thus $  \sigma_{\phi} \circ \phi $ is a cluster automorphism which is in the subgroup isomorphic to $ \mathcal{MCG}_{\bowtie}^{\pm}(S,M)  $.
	
	Suppose $ \sigma_{\phi} \circ \phi  = \phi', \phi' \in \mathcal{MCG}_{\bowtie}^{\pm}(S,M)   $ We also write $ \phi =\sigma_{\phi'} \circ \phi' $. If there are $ \phi_{1}, \phi_{2} \in  \mathcal{MCG}_{\bowtie}^{\pm}(S,M)   $ such that $ \sigma_{\phi_{1}} \circ \phi_{1}=\sigma_{\phi_{2}} \circ \phi_{2} $, then $\phi_{1} = \phi_{2}  $. It means the subgroup of $ \text{Aut}\mathcal A(S,M)  $ consisting of elements mapping $ \mathcal{T} $ to some maximal triangulation is bijective to $ \mathcal{MCG}_{\bowtie}^{\pm}(S,M) \times \mathbb{Z}_{2}  $ as sets. In fact, the bijection is an isomorphism as groups for the multiplication given by $$ (\sigma_{\phi_{1}}^{\varepsilon_{1}} \circ \phi_{1}) * (\sigma_{\phi_{2}}^{\varepsilon_{2}} \circ \phi_{2})= \sigma_{\phi_{1} \circ \phi_{2}}^{\varepsilon_{1}+\varepsilon_{2}} \circ (\phi_{1} \circ \phi_{2}), $$
	for $ ( \phi_{i},\varepsilon_{i}) \in\mathcal{MCG}_{\bowtie}^{\pm}(S,M) \times \mathbb{Z}_{2}  $.
	
	Suppose $ \phi \in 	\text{Aut}\mathcal A(S,M)  $ maps $ \mathcal{T} $ to some triangulation $ \mathcal{T}' $ which is not a maximal triangulation. Since the associated quiver $ Q(\mathcal{T}') $ contains 12 arrows if and only if $ \mathcal{T}' $ is a maximal triangulation or is the form as in the left of Figure \ref{selffoldtonofig}.
	
	After relabelling, Doing mutation along $ (6,5,2,6) $, we obtain maximal triangulation. Thus $ \mu_{\phi; 6,5,2,6} \circ \phi $ is in $ \mathcal{MCG}_{\bowtie}^{\pm}(S,M) \times \mathbb{Z}_{2} $.
	
	Suppose $  \mu_{\phi; 6,5,2,6} \circ \phi  = \phi', \phi' \in \mathcal{MCG}_{\bowtie}^{\pm}(S,M)  \times \mathbb{Z}_{2}   $ We also write $ \phi =\mu_{\phi'; 6,2,5,6} \circ \phi' $.
	
	If there are $ \phi_{1}, \phi_{2} \in  \mathcal{MCG}_{\bowtie}^{\pm}(S,M) \times \mathbb{Z}_{2}    $ such that $ \mu_{\phi_{1}; 6,2,5,6}  \circ \phi_{1}=\mu_{\phi_{2}; 6,2,5,6}  \circ \phi_{2} $, then $\phi_{1} = \phi_{2}  $.
	
	It means  $ \text{Aut}\mathcal A(S,M)  $ consisting of elements mapping $ \mathcal{T} $ to some maximal triangulation is bijective to $ \mathcal{MCG}_{\bowtie}^{\pm}(S,M) \times \mathbb{Z}_{2} \times \mathbb{Z}_{2} $ as sets. And the bijection is an isomorphism as groups for the multiplication given by
	$$ ( \mu_{\phi_{1}; 6,2,5,6}^{\varepsilon_1} \circ  \phi_{1} ) *(\mu_{\phi_{2}; 6,2,5,6}^{\varepsilon_2}  \circ \phi_{2}) =  \mu_{\phi_{1}\circ \phi_{2}; 6,2,5,6}^{\varepsilon_1+\varepsilon_2} \circ (\phi_{1}\circ \phi_{2} )$$
	for $ ( \phi_{i},\varepsilon_{i}) \in(\mathcal{MCG}_{\bowtie}^{\pm}(S,M) \times \mathbb{Z}_{2}) \times \mathbb{Z}_{2} $. Note that $ \mu_{\phi; 6,2,5,6}^{2}=1 $.
	
\end{proof}

\begin{figure}[htbp]
	\centering
	\subfigure[]{
		\begin{tikzpicture}[scale=1]
			\draw (0,0) ellipse (1.5 and 2)
			(0.6,0)--(0,2) (-0.6,0)--(0,2) (0.6,0)--(0,-2) (-0.6,0)--(0,-2) (-0.6,0)--(0.6,0)
			(0.3,0)node[above ]{$ \alpha_{1} $}
			;
			
			\draw[-latex] (-0.3,1) node[left]{$\alpha_{2}$}  -- (0.3,1)node[right]{$\alpha_{3}$} ;
			\draw[-latex] (-1.5,0)node[left]{$\alpha_{6}$}--(-0.3,1);
			\draw[-latex] (0.3,1)--(1.5,0)node[right]{$\alpha_{6}$};
			\draw[-latex] (-0.3,1) ..controls(-1,0.5) and (-1,-0.5)  .. (-0.3,-1);
			\draw[-latex] (0.3,-1) ..controls(1,-0.5) and (1,0.5)  .. (0.3,1);
			\draw[-latex] (-0.3,-1)node[left]{$\alpha_{5}$}--(-1.5,0);
			\draw[-latex] (1.5,0)--(0.3,-1)node[right]{$\alpha_{4}$};
			\draw[-latex] (0.3,-1)--(-0.3,-1);
			\draw[-latex] (-0.3,-1)--(0,0);
			\draw[-latex] (0,0)--(0.3,-1);
			\draw[-latex] (0.3,1)--(0,0);
			\draw[-latex] (0,0)--(-0.3,1);
			
			\filldraw  (0.6,0) circle [radius=1pt]
			(-0.6,0) circle [radius=1pt]
			(0,2) circle [radius=1pt]
			(0,-2) circle [radius=1pt]
			;
		\end{tikzpicture}
	}%
	\subfigure[]{
		\begin{tikzpicture}[scale=1]
			\draw (0,0) ellipse (1.5 and 2)
			(0.6,0)--(0,2) (-0.6,0)--(0,2) (0.6,0)--(0,-2) (-0.6,0)--(0,-2) (-0.6,0)--(0.6,0)
			(0.3,0)node[above ]{$ \beta_{1} $}
			;
			
			\draw[-latex] (-0.3,1) node[left]{$\beta_{4}$}  -- (0.3,1)node[right]{$\beta_{3}$} ;
			\draw[-latex] (-1.5,0)node[left]{$\beta_{6}$}--(-0.3,1);
			\draw[-latex] (0.3,1)--(1.5,0)node[right]{$\beta_{6}$};
			\draw[-latex] (-0.3,1) ..controls(-1,0.5) and (-1,-0.5)  .. (-0.3,-1);
			\draw[-latex] (0.3,-1) ..controls(1,-0.5) and (1,0.5)  .. (0.3,1);
			\draw[-latex] (-0.3,-1)node[left]{$\beta_{5}$}--(-1.5,0);
			\draw[-latex] (1.5,0)--(0.3,-1)node[right]{$\beta_{2}$};
			\draw[-latex] (0.3,-1)--(-0.3,-1);
			\draw[-latex] (-0.3,-1)--(0,0);
			\draw[-latex] (0,0)--(0.3,-1);
			\draw[-latex] (0.3,1)--(0,0);
			\draw[-latex] (0,0)--(-0.3,1);
			
			\filldraw  (0.6,0) circle [radius=1pt]
			(-0.6,0) circle [radius=1pt]
			(0,2) circle [radius=1pt]
			(0,-2) circle [radius=1pt]
			;
		\end{tikzpicture}
	}%
	\caption{The maximal triangulation in $ 4 $-punctured shpere}
	\label{sph4trifig}
\end{figure}
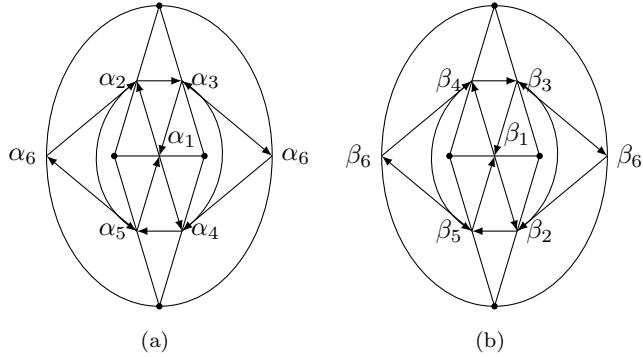

\begin{figure}[htbp]
	\centering
	\begin{tikzpicture}[scale=0.7]
		\filldraw
		(0,0) circle [radius=2pt]
		(0,2) circle [radius=2pt]
		(2,0) circle [radius=2pt]
		(-2,0) circle [radius=2pt];

		\draw (0,0) ..controls (0.7,0.5) and (1,2.75)  .. (0,2.75)node[above]{$ 2 $}
		(0,0) ..controls (-0.7,0.5) and (-1,2.75)  .. (0,2.75)
		(0,0)--node[right=-2pt]{$ 3 $}(0,2);
		
		\draw[rotate=-90] (0,0) ..controls (0.7,0.5) and (1,2.75)  .. (0,2.75)node[above]{$ 5 $}
		(0,0) ..controls (-0.7,0.5) and (-1,2.75)  .. (0,2.75)
		(0,0)--node[above=-2pt]{$ 4 $}(0,2);
		
		\draw[rotate=90] (0,0) ..controls (0.7,0.5) and (1,2.75)  .. (0,2.75)node[above]{$ 6 $}
		(0,0) ..controls (-0.7,0.5) and (-1,2.75)  .. (0,2.75)
		(0,0)--node[above=-2pt]{$ 1 $}(0,2);
		
		\draw[-latex]  (3,1)--node[above]{$ \mu_{6,5,2,6} $}(5,1);
		
		\filldraw
		(6,2) circle [radius=2pt]
		(6,0) circle [radius=2pt]
		(8,0) circle [radius=2pt]
		(8,2) circle [radius=2pt];
		
		\draw (6,2)--node[left=-2pt]{$ 5 $}(6,0) (6,2)--node[above=-2pt]{$ 2 $}(8,2) (8,0)--node[right=-2pt]{$ 3 $}(8,2) (8,0)--node[below=-2pt]{$ 4 $}(6,0) (6,2)--(8,0);
		\draw[dashed]  (6,0)--(8,2)
		(7.25,1.25)node[right]{$ 6 $}
		(6.75,1.25)node[left]{$ 1 $}
		;
	\end{tikzpicture}
	\caption{From the triangulation with three self-folded triangles to the maximal triangulation}
	\label{selffoldtonofig}
	
\end{figure}
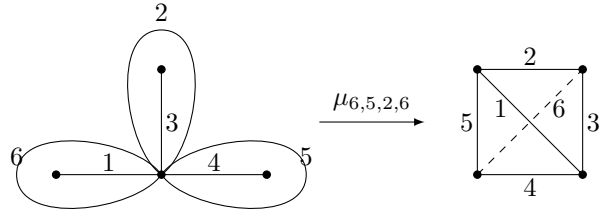

	There is a list of cluster automorphism groups $ \text{Aut}\mathcal{A}(S,M) $ of cluster algebras  $ \mathcal{A}(S,M) $ from surfaces $ (S,M) $.
	\begin{table}[H]
		\centering
		\begin{tabular}{lcc}\hline
			surfaces &   &   $ \text{Aut}\mathcal{A}(S,M) $ \\ \hline
			\multirow{2}{*}{differing from (1), (2), (3), (4)}& once-punctured closed surfaces & $ \mathcal{MCG}^{\pm}(S,M) $  \\
			& not a once-punctured closed surfaces & $ \mathcal{MCG}^{\pm}_{\bowtie}(S,M)  $  \\ \hline
			(1) the 4-punctured sphere & &   $  \mathcal{MCG}_{\bowtie}^{\pm}(S,M) \times \mathbb{Z}_{2} \times \mathbb{Z}_{2} $     \\ \hline
			(2) the once-punctured 4-gon &   &  $ Di_{4} \times \Sigma_{3}  $     \\ \hline
			(3) the twice-punctured digon.    &  &   $ \mathbb{Z} \times S_{4} \rtimes \mathbb{Z}_{2} $    \\ \hline
			\multirow{3}{*}{(4) $ S_{0,2} $ with empty punctures set } &   two boundary components contains  & \multirow{2}{*}{ $ H_{p,q} $ (see (\ref{S02mcg})) }  \\
			& $ p $, $ q $ marked points, $ p\neq q $ &  \\
			&  $ p=q $ & $ H_{p,p} \rtimes \mathbb{Z}_{2} $ (see (\ref{S02mcg}))  \\ \hline
		\end{tabular}
	\end{table}
	$\mathcal{MCG}_{\bowtie}^{\pm}(S,M)= (\mathcal{MCG}(S,M)\rtimes Z_{2}) \ltimes \mathbb{Z}_{2}^{\mathcal{P}_{n}} $, $ \mathcal{MCG}^{\pm}(S,M)=\mathcal{MCG}(S,M)\rtimes Z_{2} $,
	and the mapping class group $ \mathcal{MCG}(S,M) $ listed above admits a presentation represented in Theorem \ref{mcgSg0prop} for genus $ g=0 $ and in Theorem \ref{mcgSg1prop} for genus $ g \ge 1 $.

\end{document}